\documentclass[12pt]{article}

\usepackage{amsfonts,amsmath}

\usepackage{epsfig}
\usepackage{enumerate}
\usepackage{color}

\graphicspath{{figures/}}
\usepackage{booktabs}

\newtheorem{theorem}{Theorem}[section]
\newtheorem{lemma}[theorem]{Lemma}

\newtheorem{corollary}[theorem]{Corollary}

\newtheorem{problem}[theorem]{Problem}

\numberwithin{equation}{section}

\newlength{\sperr}
\def\BBox{\hbox{\vrule height 6pt depth 0pt width 6pt}}
\newenvironment{proof}{{\settowidth{\sperr}{\rm Proof}
\parbox[t]{1.3\sperr}{\bf Proof.}
}~}{\nopagebreak\mbox{}\hfill$\BBox$\par\addvspace{4mm}}

\newcommand{\bzero}{\mbox{\boldmath{$0$}}}

\newcommand{\fb}{\mbox{\boldmath{$f$}}}

\newcommand{\bvarepsilon}{\mbox{\boldmath{$\varepsilon$}}}

\newcommand{\bu}{\mbox{\boldmath{$u$}}}

\newcommand{\bI}{\mbox{\boldmath{$I$}}}

\newcommand{\blambda}{\mbox{\boldmath{$\lambda$}}}

\newcommand{\bv}{\mbox{\boldmath{$v$}}}
\newcommand{\bV}{\mbox{\boldmath{$V$}}}

\newcommand{\bz}{\mbox{\boldmath{$z$}}}

\newcommand{\bsigma}{\mbox{\boldmath{$\sigma$}}}

\newcommand{\bnu}{\mbox{\boldmath{$\nu$}}}

\newcommand{\real}{\mbox{$\mathbb{R}$}}

\begin{document}

\begin{center}
\LARGE\bf Well-posedness and Numerical Analysis of\\
Mixed Variational-hemivariational Inequalities
\end{center}

\begin{center}
\large
Weimin Han,\footnote{Department of Mathematics, University of Iowa, Iowa City, IA 52242-1410, USA;
email: weimin-han@uiowa.edu}\quad
Jianguo Huang\footnote{School of Mathematical Sciences, and MOE-LSC, Shanghai Jiao Tong University, Shanghai 200240, China; email: jghuang@sjtu.edu.cn.  The work of this author was partially supported by NSFC (Grant no.\ 12571390).}
\quad and\quad
Yuan Yao\footnote{Program in Applied Mathematical and Computational Sciences, University of Iowa, Iowa City, IA 52242-1410, USA; email: yuan-yao@uiowa.edu}
\end{center}

\smallskip

\medskip
\begin{quote}
{\bf Abstract.} 
The paper is devoted to well-posedness analysis and the numerical solution of a family of general elliptic mixed variational-hemivariational
inequalities.  Various mixed variational equations, mixed variational inequalities and mixed 
hemivariational inequalities found in the literature are special cases of the mixed 
variational-hemivariational inequalities. Well-posedness of the mixed variational-hemivariational inequalities and their numerical 
approximations are studied via the projection iteration technique.  Error analysis of the numerical
methods is presented.  The results are applied to the study of a variational-hemivariational inequality
of the Stokes equations for incompressible fluid flows subject to slip conditions of frictional type, 
both monotone and non-monotone.  Optimal order error estimates are derived for the use of some stable 
finite element space pairs under certain solution regularity assumptions.  Numerical results are 
reported demonstrating the theoretical prediction of convergence orders.
\end{quote}



\medskip

\section{Introduction}\label{sec:intro}

Mixed formulations arise frequently for application problems involving constraints.
A well-known example of mixed formulations occurs in the treatment of the incompressibility constraint for fluid flow problems.
Moreover, mixed formulations are useful in developing efficient numerical methods for the computation of
physical quantities other than the original unknown variable of the underlying partial differential equations.
To solve mixed problems from applications, numerical methods are needed.  There is a large number of
publications on the numerical solution of mixed problems.
The book \cite{BBF2013} is a standard reference on mixed finite element methods for solving a variety of boundary value problems through
their mixed formulations.  The majority of publications on mixed numerical methods are devoted to solving mixed variational equations.  Nevertheless, some references are available on mixed finite element methods for solving mixed variational inequalities, e.g., \cite{BHR78, HHNL1988, SBL04}.  

While mixed variational inequalities describe mathematical models of non-smooth phenomena in which the 
non-smooth functional is assumed to be convex, mixed hemivariational inequalities are mathematical models
for non-smooth problems in which the non-smooth functional can be non-convex.  A number of papers can be 
found on well-posedness analysis of mixed hemivariational inequalities, or that of more general mixed 
variational-hemivariational inequalities.  In \cite{BMZ19, BMZ20, Mat19, MBZ19}, solution existence results
for mixed variational-hemivariational inequalities are proved by applying abstract fixed-point principles for set-valued mappings.  In contrast, in \cite{HM22a}, for a particular mixed 
variational-hemivariational inequality of rank $(1,1)$ (following the nomenclature in 
\cite[Section 5.1, Chapter 7]{Han2024}, cf.\ Section \ref{sec:well}), a minimax principle is 
established and the inequality problem is studied through an equivalent saddle-point formulation. 
In \cite{HM22b}, a combination of the results established in \cite{HM22a} and the Banach fixed-point argument leads to well-posedness results of 
general mixed variational-hemivariational inequalities of rank $(1,1)$ and rank $(2,1)$.  
In \cite{TZ25}, a projection iteration technique is developed in the study of a mixed hemivariational
inequality, which can be termed a mixed variational-hemivariational inequality of rank $(0,1)$ 
following the nomenclature in \cite{Han2024}.  In this paper, we extend the projection iteration technique 
for the well-posedness analysis of general mixed variational-hemivariational inequalities of rank 
$(1,1)$ and rank $(2,1)$.  Compared to the results in \cite{HM22a, HM22b}, uniqueness of both unknown 
variables (uniqueness of both velocity and pressure in applications of incompressible fluid flows) 
can be established in the abstract level, at the expense of a finer setting of the
mixed variational-hemivariational inequalities.  We then proceed to perform numerical analysis of the general mixed variational-hemivariational inequalities.  

Hemivariational inequalities are useful for applications involving non-smooth, non-monotone and set-valued relations among
physical quantities.  Since the pioneering work of Panagiotopoulos four decades ago (\cite{Pa83}), modeling, analysis, numerical
solution and applications of hemivariational inequalities have attracted more and more attention from the research community. 
One has witnessed a substantial increase in the number of publications related to hemivariational inequalities in recent years. 
As representative recent references for well-posedness analysis results of variational-hemivariational 
inequalities, one is referred to \cite{SM2025} based on applications of abstract surjectivity results on pseudomonotone 
operators and to \cite{Han2024} for an accessible approach based on basic knowledge from Functional Analysis.  One is also referred 
to \cite{HS19AN, HFWH25} for recent surveys on numerical analysis of variational-hemivariational inequalities.

The salient features of the paper are that (1) the mixed variational-hemivariational inequalities 
studied are rather general, and they include a variety of mixed formulations found in the literature as
special cases (cf.\ the comments right after the statement of Problem \ref{p3} in Section \ref{sec:well}); 
(2) analysis of well-posedness of the mixed variational-hemivariational inequalities and their numerical 
approximations is carried out by an easily accessible approach, without the need of abstract theories 
of pseudomonotone operators used in many references on hemivariational inequalities or
variational-hemivariational inequalities.  

In the study of hemivariational inequalities, we need the notions of the generalized directional derivative and 
generalized subdifferential in the sense of Clarke (\cite{Cl1983}).  
Let $\Psi\colon V\to \real$ be a locally Lipschitz continuous functional defined on a real Banach space $V$.  
Then its generalized (Clarke) directional derivative at $u\in V$ in the direction $v \in V$ is defined by
\[ \Psi^0(u; v) := \limsup_{w\to u,\,\lambda\downarrow 0}\frac{\Psi(w+\lambda v) -\Psi(w)}{\lambda}, \]
and the generalized subdifferential of $\Psi$ at $u\in V$ is defined by
\[ \partial\Psi(u) := \left\{\eta\in V^* \mid \Psi^{0}(u;v)\ge\langle\eta,v\rangle\ \forall\,v \in V\right\}.  \]

In the special case where $\Psi \colon V \to \real$ is locally Lipschitz continuous and convex, the
subdifferential $\partial \Psi(u)$ at any $u \in V$ in the sense of Clarke coincides with the convex subdifferential
$\partial \Psi(u)$.  Hence, the notion of the Clarke subdifferential can be viewed as a generalization of that
of the convex subdifferential.  For all $\lambda \in \real$ and all $u\in V$, we have
\[ \partial(\lambda\,\Psi)(u)=\lambda\,\partial \Psi(u).  \]
Moreover, for locally Lipschitz functions $\Psi_1,\Psi_2 \colon V \to \real$, the inclusion
\begin{equation}\label{prop1}
\partial (\Psi_1 + \Psi_2) (u) \subset \partial \Psi_1 (u) + \partial \Psi_2 (u)\quad\forall\,u\in V
\end{equation}
holds, which is equivalent to the inequality
\begin{equation}\label{prop2}
(\Psi_1 + \Psi_2)^0(u; v) \le \Psi_1^0(u; v) + \Psi_2^0(u; v)\quad\forall\,u,v\in V.
\end{equation}

Detailed discussions of the generalized directional derivative and the generalized subdifferential for locally Lipschitz
continuous functionals, including their properties, can be found in several references, e.g.\ \cite{Cl1983, MOS2013}.

The rest of the paper is organized as follows.   In Section \ref{sec:well}, we address the well-posedness of mixed variational-hemivariational 
inequalities, first for those of rank $(2,1)$, and then for those of rank $(1,1)$.  In Section \ref{sec:na}, 
we provide numerical analysis of the mixed variational-hemivariational inequalities.
In Section \ref{sec:Stokes}, we apply the results presented in Sections \ref{sec:well} and \ref{sec:na}
in the study of a variational-hemivariational inequality for the Stokes equations subject to 
non-leak slip boundary conditions of friction type.  In Section \ref{sec:ex}, we report computer simulation results on some numerical examples.

\section{Well-posedness of mixed variational-hemivariational inequalities}\label{sec:well}

From now on, we let $V$ and $Q$ be two real Hilbert spaces.  Their dual spaces are denoted by
$V^*$ and $Q^*$.  The symbol $\langle\cdot,\cdot\rangle$ denotes the duality pairing between $V^*$ and $V$,
or between $Q^*$ and $Q$;  it should be clear from the context which duality pairing is meant by
$\langle\cdot,\cdot\rangle$.  Let $K_V\subset V$ and $K_Q\subset Q$.
Let $A\colon V\to V^*$, $b\colon V\times Q\to \real$, $\Phi\colon V\times V\to\real$, 
$\Psi\colon V\to \real$, and $f\in V^*$ be given operators and functionals.  We consider a general 
mixed variational-hemivariational inequality.

\begin{problem}\label{p3}
Find $(u,p)\in K_V \times K_Q$ such that
\begin{align}
&\langle Au,v-u\rangle +b(v-u,p)+\Phi(u,v)-\Phi(u,u) +\Psi^0(u;v-u)\ge\langle f,v-u\rangle
\quad\forall\,v\in K_V,\label{n1}\\
&b(u,q-p)\leq 0\quad\forall\, q\in K_Q.\label{n2}
\end{align}
\end{problem}

We comment that the form of Problem \ref{p3} is rather general, and it includes a variety of mixed
problems found in the literature as special cases.  E.g., the mixed problem studied in \cite{SBL04}
corresponds to the choice $\Phi\equiv 0$, $\Psi\equiv 0$, and $A$ linear (for the mixed problem 
in \cite{SBL04}, the right hand side of \eqref{n2} is replaced by a continuous linear functional, which 
does not cause any complication in the analysis and numerical solution of the problem).  For the mixed 
problem studied in \cite{HKSS17} and the Stokes variational inequality studied in \cite{DK16}, $A$ is 
linear, $\Psi\equiv 0$, $K_V=V$ and $K_Q=Q$, the latter having the implication that \eqref{n2} is reduced  
to a simpler equation. The problem studied in \cite{KHSJ18} is simpler than that in \cite{HKSS17} in the
sense that $\Phi$ depends on only one variable. For the problem studied in \cite{FCHCD20} and 
\cite{LHZ22}, $A$ is linear, $\Phi\equiv 0$, $K_V=V$ and $K_Q=Q$. In the very special case where 
$\Phi\equiv 0$, $\Psi\equiv 0$, $K_V=V$, $K_Q=Q$, and $A$ linear, Problem \ref{p3} reduces to the 
commonly seen mixed variational equations extensively studied in the literature (cf.\ \cite{BBF2013}).

In the study of Problem \ref{p3}, we will use the following conditions on the problem data.

\begin{itemize}
\item $H(K_V)$ \ $V$ is a real Hilbert space, $K_V\subset V$ is non-empty, closed and convex, and it contains a subspace $V_0$.

\item $H(K_Q)$ \ $Q$ is a real Hilbert space, $K_Q\subset Q$ is non-empty, closed and convex.

\item  $H(A)$ \ $A\colon V\to V^*$ is Lipschitz continuous and strongly monotone.

\item  $H(b)$ \ $b:V\times Q\to \mathbb R$ is bilinear and bounded.  Moreover, the inf-sup condition holds:
\begin{equation}
\alpha_b \|q\|_Q\le\sup_{w\in V_0} \frac{b(w,q)}{\|w\|_V}\quad\forall\,q\in Q. 
\label{infsup}
\end{equation}

\item $H(\Phi)_2$ \ $\Phi\colon V\times V\to\mathbb{R}$; for any $u\in V$, $\Phi(u,\cdot)\colon V\to\mathbb{R}$
is convex and bounded above on a non-empty open set; and there exists a constant $\alpha_\Phi\ge 0$ such that
\begin{equation}
\Phi(u_1,v_2)-\Phi(u_1,v_1)+\Phi(u_2,v_1)-\Phi(u_2,v_2)\le \alpha_\Phi \|u_1-u_2\|\,\|v_1-v_2\|
\quad\forall\,u_1,u_2,v_1,v_2\in V.
\label{n3}
\end{equation}
Moreover, 
\begin{equation}
\Phi(u,v+w)=\Phi(u,v)\quad\forall\,u,v\in K_V,\,w\in V_0
\label{Phi2}
\end{equation}

\item  $H(\Psi)$ \ $\Psi\colon V\to\mathbb{R}$ is locally Lipschitz continuous, and there exists a constant
$\alpha_\Psi\geq 0$ such that
\begin{align}
\Psi^0(v_1;v_2-v_1)+\Psi^0(v_2;v_1-v_2)\le \alpha_\Psi\|v_1-v_2\|_V^2\quad\forall\,v_1,v_2\in V.\label{eq4}
\end{align}
Moreover,
\begin{equation}
\Psi^0(v;w)=0\quad\forall\,v\in K_V,\,w\in V_0.
\label{Psi}
\end{equation}

\item  $H(f)$ \ $f\in V^*$.

\end{itemize}

Following the nomenclature used in \cite{Han2024}, Problem \ref{p3} is a mixed variational-hemivariational 
inequality of rank $(2,1)$ since the function $\Phi$ depends on two variables and the function $\Psi$ 
depends on one variable. When $\Psi\equiv 0$, Problem \ref{p3} is reduced to a mixed variational inequality.
When $\Phi\equiv 0$, Problem \ref{p3} can be viewed as a ``pure'' mixed hemivariational inequality.

We comment that under the assumption $H(K_V)$, $v+w\in K_V$ for any $v\in K_V$ and any $w\in V_0$
(\cite[Lemma 2.1]{SBL04}).

Related to the condition $H(A)$, we will use $M_A>0$ for the Lipschitz constant:
\begin{equation}
 \|Av_1 -Av_2\|_{V^*} \le M_A \|v_1-v_2\|_V\quad\forall\,v_1,v_2\in V,
\label{Lip:A}
\end{equation}
and use $m_A>0$ for the strong monotonicity constant:
\begin{equation}
 \langle Av_1 -Av_2, v_1-v_2\rangle \ge m_A \|v_1-v_2\|_V^2\quad\forall\,v_1,v_2\in V.
\label{Mon:A}
\end{equation}
Related to the condition $H(b)$,  we will use $M_b>0$ for the boundedness constant:
\begin{equation}
 \left|b(v,q)\right| \le M_b \|v\|_V \|q\|_Q \quad\forall\,v\in V,\,q\in Q.
\label{bd:b}
\end{equation}
Assumption $H(b)$ allows us to define an operator $B\in {\cal L}(V;Q)$ by  the relation
\[ (Bv,q)_Q = b(v,q)\quad\forall\,v\in V,\,q\in Q.\]
The inequality \eqref{bd:b} is equivalent to 
\[ \|B\|\le M_b. \]

Note that the assumption $\Phi(u,\cdot)\colon V\to\mathbb{R}$ being convex and bounded above on a non-empty open set
implies that $\Phi(u,\cdot)$ is continuous on $V$ (cf.\ \cite{ET}).
The convex function $\Phi(u,\cdot)\colon V\to\mathbb{R}$ is assumed to be continuous, instead of l.s.c.  
As is explained in \cite{Han20, Han21} or \cite[Section 5.1]{Han2024}, there is no loss of generality with 
the stronger assumption of continuity for a vast majority of applications. The subscript 2 in $H(\Phi)_2$ reminds the reader
that this is a condition for the case where $\Phi$ depends on two variables.  The constant $\alpha_\Psi\ge 0$ 
in \eqref{eq4} measures the strength of non-convexity of $\Psi$: the smaller the value of $\alpha_\Psi$, 
the weaker the non-convexity of $\Psi$; in particular, in the degenerate case where $\Psi$ is convex, \eqref{eq4} holds with $\alpha_\Psi=0$.

The first step in the analysis of Problem \ref{p3} is to convert the inequality \eqref{n2} in the form of
an equality. Let $P_{K_Q}\colon Q\to K_Q$ be the orthogonal projection operator from $Q$ to $K_Q$.  

\begin{lemma}\label{lem:b}
Under the assumption $H(b)$, \eqref{n2} is equivalent to 
\begin{equation}
p=P_{K_Q}(p+\rho\,Bu) \quad\forall\,\rho>0.
\label{n2p}
\end{equation}
\end{lemma}
\begin{proof}
By making use of the operator $B$, we can rewrite \eqref{n2} as
\[ (Bu,q-p)_Q\le 0\quad\forall\,q\in Q.\]
For any $\rho>0$, this is equivalent to
\[ ((p+\rho\,Bu)-p,q-p)_Q\le 0\quad\forall\,q\in Q,\]
i.e., \eqref{n2p} holds. 
\end{proof}

\begin{theorem}\label{thm:main2}
Assume $H(K_V)$, $H(K_Q)$, $H(A)$, $H(b)$, $H(\Phi)_2$, $H(\Psi)$, $H(f)$, and 
\begin{equation}
\alpha_\Phi+\alpha_\Psi<m_A.
\label{small2}
\end{equation}
Then, Problem \ref{p3} has a unique solution $(u,p)\in K_V\times K_Q$.
\end{theorem}
\begin{proof}
Let $p_0\in V_Q$ be arbitrary but fixed.  The parameter $\rho>0$ will be chosen sufficiently small later in the proof.  Define a sequence $\{(u_n,p_n)\}_{n\ge 0}\subset K_V\times K_Q$ by 
\begin{align}
&\langle Au_n,v-u_n\rangle +\Phi(u_n,v)-\Phi(u_n,u_n) +\Psi^0(u_n;v-u_n)\nonumber\\
&\qquad \ge\langle f,v-u_n\rangle-b(v-u_n,p_n)\quad\forall\,v\in K_V,\label{n1b}\\
&p_{n+1}=P_{K_Q}(p_n+\rho\,Bu_n). \label{n2b}
\end{align}
Under the stated assumptions, given $p_n\in K_Q$, the variational-hemivariational inequality \eqref{n1b} has a unique solution $u_n\in K_V$ (cf.\ e.g., \cite[Theorem 5.9]{Han2024}).  We take $v=u_{n-1}$ in \eqref{n1b}
and take $v=u_n$ in \eqref{n1b} with $n$ replaced by $(n-1)$:
\begin{align*}
&\langle Au_n,u_{n-1}-u_n\rangle +\Phi(u_n,u_{n-1})-\Phi(u_n,u_n) +\Psi^0(u_n;u_{n-1}-u_n)\\
&\qquad \ge\langle f,u_{n-1}-u_n\rangle-b(u_{n-1}-u_n,p_n)\\
&\langle Au_{n-1},u_n-u_{n-1}\rangle+\Phi(u_{n-1},u_n)-\Phi(u_{n-1},u_{n-1})+\Psi^0(u_{n-1};u_n-u_{n-1})\\
&\qquad \ge\langle f,u_n-u_{n-1}\rangle-b(u_n-u_{n-1},p_{n-1}).
\end{align*}
Add the two inequalities to get
\begin{align*}
\langle Au_n-Au_{n-1},u_n-u_{n-1}\rangle
& \le \Phi(u_n,u_{n-1})-\Phi(u_n,u_n)+\Phi(u_{n-1},u_n)-\Phi(u_{n-1},u_{n-1})\\
&\quad{} +\Psi^0(u_n;u_{n-1}-u_n)+\Psi^0(u_{n-1};u_n-u_{n-1})\\
&\quad{} -b(u_n-u_{n-1},p_n-p_{n-1}).
\end{align*}
Applying $H(A)$, $H(\Phi)_2$ and $H(\Psi)$,
\[ m_A \|u_n-u_{n-1}\|_V^2 \le \alpha_\Phi \|u_n-u_{n-1}\|_V^2 + \alpha_\Psi \|u_n-u_{n-1}\|_V^2
-b(u_n-u_{n-1},p_n-p_{n-1}).\]
Hence,
\begin{equation}
\left(m_A-\alpha_\Phi-\alpha_\Psi\right) \|u_n-u_{n-1}\|_V^2 \le -b(u_n-u_{n-1},p_n-p_{n-1}).
\label{n20}
\end{equation}
We derive from \eqref{n20} that
\[ \left(m_A-\alpha_\Phi-\alpha_\Psi\right) \|u_n-u_{n-1}\|_V^2 \le M_b \|u_n-u_{n-1}\|_V \|p_n-p_{n-1}\|_Q, \]
i.e.,
\begin{equation}
\|u_n-u_{n-1}\|_V\le c_0\|p_n-p_{n-1}\|_Q,\quad c_0:=\frac{M_b}{m_A-\alpha_\Phi-\alpha_\Psi}.
\label{n21}
\end{equation}

Next, from the definition \eqref{n2b} for $p_{n+1}$ and $p_n$, and the non-expansiveness of $P_{K_Q}$, we have
\begin{align*}
\|p_{n+1}-p_n\|_Q^2 & \le \|(p_n+\rho\,Bu_n) - (p_{n-1}+\rho\,Bu_{n-1})\|_Q^2\\
& = \|p_n-p_{n-1}\|_Q^2 + \rho^2 \|B(u_n-u_{n-1})\|_Q^2 +2\,\rho\,b(u_n-u_{n-1},p_n-p_{n-1}). 
\end{align*}
From \eqref{n20},
\[ b(u_n-u_{n-1},p_n-p_{n-1}) \le -\left(m_A-\alpha_\Phi-\alpha_\Psi\right) \|u_n-u_{n-1}\|_V^2. \]
Also, 
\[ \|B(u_n-u_{n-1})\|_Q \le M_b \|u_n-u_{n-1}\|_V. \]
Thus,
\begin{align*}
\|p_{n+1}-p_n\|_Q^2 & \le \|p_n-p_{n-1}\|_Q^2 + \rho^2 M_b^2 \|u_n-u_{n-1}\|_V^2
-2\,\rho\,\left(m_A-\alpha_\Phi-\alpha_\Psi\right) \|u_n-u_{n-1}\|_V^2\\
& = \|p_n-p_{n-1}\|_Q^2 - \theta\,\|u_n-u_{n-1}\|_V^2,
\end{align*}
where
\[ \theta = 2\,\rho\,\left(m_A-\alpha_\Phi-\alpha_\Psi\right) - \rho^2 M_b^2. \]
Choose $\rho>0$ sufficiently small so that $\theta>0$.  Then, 
\begin{equation}
\|p_{n+1}-p_n\|_Q^2\le \|p_n-p_{n-1}\|_Q^2 - \theta\,\|u_n-u_{n-1}\|_V^2.
\label{n22}
\end{equation}

By making use of the assumptions \eqref{Phi2} and \eqref{Psi}, we derive from \eqref{n1b} the following equality
\begin{equation}
\langle Au_n,w\rangle + b(w,p_n) =\langle f,w\rangle  \quad\forall\,w\in V_0. 
\label{n23}
\end{equation}
Hence,
\[  b(w,p_{n}-p_{n-1}) = - \langle Au_{n}-Au_{n-1},w\rangle \quad\forall\,w\in V_0. \]
Applying the inf-sup condition \eqref{infsup},
\[ \alpha_b \|p_{n}-p_{n-1}\|_Q \le \sup_{w\in V_0} \frac{b(w,p_{n}-p_{n-1})}{\|w\|_V}
= \sup_{w\in V_0} \frac{-\langle Au_{n}-Au_{n-1},w\rangle}{\|w\|_V}\le M_A \|u_{n}-u_{n-1}\|_V, \]
or,
\begin{equation}
\|u_{n}-u_{n-1}\|_V \ge \frac{\alpha_b}{M_A}\,\|p_{n}-p_{n-1}\|_Q.
\label{n24}
\end{equation}
Combining \eqref{n22} and \eqref{n24},
\[ \|p_{n+1}-p_n\|_Q \le \kappa\,\|p_n-p_{n-1}\|_Q.\]
where for $\theta>0$ sufficiently small, which is guaranteed if $\rho>0$ is sufficiently small, we have
\[ \kappa:=\left(1-\theta\,\frac{\alpha_b^2}{M_A^2}\right)^{1/2}<1.\]
Then,
\begin{equation}
\|p_{n+1}-p_n\|_Q \le \kappa^n \|p_1-p_0\|_Q.
\label{n25}
\end{equation}
For $m>n$, we apply \eqref{n25} repeatedly to derive
\begin{align*}
\|p_m-p_n\|_Q & \le \|p_m-p_{m-1}\|_Q+\cdots +\|p_{n+1}-p_n\|_Q\\
&\le \left(\kappa^{m-1}+\cdots+\kappa^n\right) \|p_1-p_0\|_Q\\
&\le \frac{\kappa^n}{1-\kappa}\,\|p_1-p_0\|_Q.
\end{align*}
Since $\kappa<1$, we have 
\[ \|p_m-p_n\|_Q\to 0\quad{\rm as}\ m,n\to\infty,\]
i.e., $\{p_n\}$ is a Cauchy sequence.  Since $Q$ is complete, there exists an element $p\in Q$
such that 
\[ p_n\to p\quad{\rm in}\ Q.\]
Moreover, since $K_Q$ is closed, the limit $p\in K_Q$.

By \eqref{n21} and \eqref{n25},
\[ \|u_n-u_{n-1}\|_V\le c_0 \kappa^{n-1} \|p_1-p_0\|_Q.\]
Hence, for $m>n$,
\[ \|u_m-u_n\|_V \le \|u_m-u_{m-1}\|_V+\cdots+\|u_{n+1}-u_n\|_V\le \frac{c_0 \kappa^n}{1-\kappa}\,\|p_1-p_0\|_Q.\]
So $\{u_n\}$ is a Cauchy sequence in the complete space $V$, and thus has a limit $u\in V$.  Since 
$u_n\in K_V$ and $K_V$ is closed, $u\in K_V$.  

Note that  
\begin{align*}
& \left| \left(\Phi(u_n,v)-\Phi(u_n,u_n)\right) - \left(\Phi(u,v)-\Phi(u,u)\right) \right| \\
&\qquad \le \left| \Phi(u_n,v)-\Phi(u_n,u_n) - \Phi(u,v)+\Phi(u,u_n) \right| 
+ \left| \Phi(u,u)-\Phi(u,u_n)\right|  \\
&\qquad \le \alpha_\Phi\|u_n-u\|_V \|v-u_n\|_V + \left| \Phi(u,u)-\Phi(u,u_n)\right|  \\
&\qquad \to 0,
\end{align*}
and (cf.\ \cite[Proposition 3.23\,(ii)]{MOS2013})
\[ \limsup_{n\to\infty} \Psi^0(u_n;v-u_n) \le \Psi^0(u;v-u). \]
We take the upper limit of both sides of \eqref{n1b} to obtain \eqref{n1}.

Due to the continuity of $P_{K_Q}$, we take the limit in \eqref{n2b} to obtain
\[ p=P_{K_Q}(p+\rho\,Bu). \]
By Lemma \ref{lem:b}, this implies \eqref{n2}.

Hence, the limit $(u,p)\in K_V\times K_Q$ is a solution of Problem \ref{p3}. 

For uniqueness, let $(u_1,p_1)\in K_V\times K_Q$ and $(u_2,p_2)\in K_V\times K_Q$ be two solutions of 
Problem \ref{p3}.  Then, 
\begin{align}
&\langle Au_1,u_2-u_1\rangle +b(u_2-u_1,p_1)+\Phi(u_1,u_2)-\Phi(u_1,u_1) +\Psi^0(u_1;u_2-u_1)\ge\langle f,u_2-u_1\rangle, \label{n26}\\
&b(u_1,p_2-p_1)\leq 0 \label{n27}
\end{align}
and
\begin{align}
&\langle Au_2,u_1-u_2\rangle +b(u_1-u_2,p_2)+\Phi(u_2,u_1)-\Phi(u_2,u_2) +\Psi^0(u_2;u_1-u_2)\ge\langle f,u_1-u_2\rangle, \label{n28}\\
&b(u_2,p_1-p_2)\leq 0. \label{n29}
\end{align}
Add \eqref{n27} and \eqref{n29} to get
\begin{equation}
 -b(u_1-u_2,p_1-p_2)\leq 0. 
\label{n30}
\end{equation}
Then, add \eqref{n26} and \eqref{n28} to get
\begin{align}
\langle Au_1-Au_2,u_1-u_2\rangle
&\le \Phi(u_1,u_2)-\Phi(u_1,u_1)+\Phi(u_2,u_1)-\Phi(u_2,u_2) \nonumber\\
&\quad{}  +\Psi^0(u_1;u_2-u_1)+\Psi^0(u_2;u_1-u_2)-b(u_1-u_2,p_1-p_2).
\label{n30a}
\end{align}
Hence, by making use of $H(A)$, $H(\Phi)_2$, $H(\Psi)$ and \eqref{n30},
\begin{equation}
m_A \|u_1-u_2\|_V^2 \le \alpha_\Phi \|u_1-u_2\|_V^2 + \alpha_\Psi \|u_1-u_2\|_V^2, 
\label{n30b}
\end{equation}
i.e.,
\[ \left(m_A-\alpha_\Phi-\alpha_\Psi\right) \|u_1-u_2\|_V^2 \le 0. \]
By the smallness condition \eqref{small2}, we deduce that $\|u_1-u_2\|_V=0$ and $u_1=u_2$.

Similar to \eqref{n24}, we have
\[ \|p_1-p_2\|_Q\le \frac{\alpha_b}{M_A}\,\|u_1-u_2\|_V = 0, \]
implying the uniqueness of $p$: $p_1=p_2$.
\end{proof}

Now we present a Lipschitz continuous dependence property of the solution on the right side.

\begin{theorem}\label{thm:main3}
Keep the assumptions stated in Theorem \ref{thm:main2}. Then, there exists a constant $c>0$ such that
for solutions $(u_1,p_1),(u_2,p_2)\in K_V\times K_Q$ of Problem \ref{p3} with $f=f_1,f_2\in V^*$, 
the following inequality holds:
\begin{equation}
\|u_1-u_2\|_V + \|p_1-p_2\|_Q \le c\,\|f_1-f_2\|_{V^*}.
\label{n11}
\end{equation}
\end{theorem}
\begin{proof}
We start with \eqref{n26}--\eqref{n29} with $f$ replaced by $f_1$ in \eqref{n26} and by $f_2$ in \eqref{n28}.  Then, \eqref{n30a} is modified to
\begin{align*}
\langle Au_1-Au_2,u_1-u_2\rangle
&\le \Phi(u_1,u_2)-\Phi(u_1,u_1)+\Phi(u_2,u_1)-\Phi(u_2,u_2)\\
&\quad{}  +\Psi^0(u_1;u_2-u_1)+\Psi^0(u_2;u_1-u_2)\\
&\quad{} -b(u_1-u_2,p_1-p_2) + \langle f_1-f_2,u_1-u_2\rangle
\end{align*}
whereas \eqref{n30b} is modified to
\[ m_A \|u_1-u_2\|_V^2 \le \alpha_\Phi \|u_1-u_2\|_V^2 + \alpha_\Psi \|u_1-u_2\|_V^2
+\|f_1-f_2\|_{V^*}\|u_1-u_2\|_V . \]
Then
\begin{equation}
\|u_1-u_2\|_V \le \frac{1}{m_A - \alpha_\Phi -\alpha_\Psi}\,\|f_1-f_2\|_{V^*}.
\label{n30c}
\end{equation}

Similar to \eqref{n23}, we have
\begin{align*}
& \langle Au_1,w\rangle + b(w,p_1) =\langle f_1,w\rangle  \quad\forall\,w\in V_0,\\
& \langle Au_2,w\rangle + b(w,p_2) =\langle f_2,w\rangle  \quad\forall\,w\in V_0. 
\end{align*}
Subtract the two equations to get
\[ b(w,p_1-p_2) =\langle f_1-f_2,w\rangle - \langle Au_1-Au_2,w\rangle \quad\forall\,w\in V_0. \]
By the inf-sup condition \eqref{infsup}, 
\[ \alpha_b\|p_1-p_2\|_Q \le \sup_{w\in V_0} \frac{1}{\|w\|_V}\left[ \langle f_1-f_2,w\rangle - \langle Au_1-Au_2,w\rangle \right]
\le \|f_1-f_2\|_{V^*} + M_A \|u_1-u_2\|_V. \]
This inequality and \eqref{n30c} together imply \eqref{n11}. 
\end{proof}

We now introduce a variant of Problem \ref{p3} which is directly related to applications.
To simplify the notation, for $\Delta$ a subset of $\Omega$ or a subset of its boundary, we use 
$I_\Delta(g)$ to denote the integral of an integrable function $g$ on $\Delta$.  For application 
problems, very often, the mixed  variational-hemivariational inequality of rank $(2,1)$ is in the form
of Problem \ref{p3} in which $\Psi^0(u;v-u)$ is replaced by an integral over $\Delta$: $I_\Delta(\psi^0(\gamma u;\gamma v-\gamma u))$. 
Here, for a positive integer $m$, $\psi\colon \mathbb{R}^m\to \mathbb{R}$ is locally Lipschitz continuous
and for a function space $V_\psi$ on $\Delta$, the operator $\gamma\colon V\to V_\psi$ is such that
$\psi^0(\gamma u;\gamma v)\in L^1(\Delta)$ for all $u,v\in V$.  In other words, for application problems,
instead of Problem \ref{p3}, we have its following variant.

\begin{problem}\label{p3v}
Find $(u,p)\in K_V \times K_Q$ such that
\begin{align}
&\langle Au,v-u\rangle +b(v-u,p)+\Phi(u,v)-\Phi(u,u)+I_\Delta(\psi^0(\gamma u;\gamma v-\gamma u))\nonumber\\
&\qquad  \ge\langle f,v-u\rangle \quad\forall\,v\in K_V,\label{n1v}\\
&b(u,q-p)\leq 0\quad\forall\, q\in K_Q.\label{n2v}
\end{align}
\end{problem}

For mixed variational-hemivariational inequalities in applications, associated with second-order PDEs, we let 
$V_\psi=L^2(\Delta)$, and denote by $c_\Delta>0$ the smallest constant in the inequality
\begin{equation}
I_\Delta(|\gamma v|_{\mathbb{R}^m}^2)\le c_\Delta^2 \|v\|_V^2\quad\forall\,v\in V.
\label{cd}
\end{equation}
We replace $H(\Psi)$ by the following assumptions on $\psi$.

\begin{itemize}
\item $H(\psi)$ \ For an integer $m\ge 1$, $\psi\colon \mathbb{R}^m\to \mathbb{R}$ is locally Lipschitz continuous, and there exist constants $c_\psi\ge 0$ and $\alpha_\psi\ge 0$ such that
\begin{align}
& |\partial \psi(z)|_{\mathbb{R}^m} \le c_\psi \left(1+|z|_{\mathbb{R}^m}\right)\quad\forall\,z\in\mathbb{R}^m, \label{psi1} \\
& \psi^0(z_1;z_2-z_1)+\psi^0(z_2;z_1-z_2)\le\alpha_\psi |z_1-z_2|_{\mathbb{R}^m}^2 \quad\forall\,z_1,z_2\in\mathbb{R}^m. \label{psi2}
\end{align}
\end{itemize}

We deduce from Theorems \ref{thm:main2} and \ref{thm:main3} the next result on Problem \ref{p3v}.

\begin{theorem}\label{thm:main3v}
Assume $H(K_V)$, $H(K_Q)$, $H(A)$, $H(b)$, $H(\Phi)_2$, $H(\psi)$, $H(f)$, and 
\begin{equation}
\alpha_\Phi+\alpha_\psi c_\Delta^2<m_A.
\label{small3v}
\end{equation}
Then, Problem \ref{p3v} has a unique solution $(u,p)\in K_V\times K_Q$.  Moreover, there exists a constant $c>0$ such that
for solutions $(u_1,p_1),(u_2,p_2)\in K_V\times K_Q$ of Problem \ref{p3v} with $f=f_1,f_2\in V^*$, 
the following inequality holds:
\[ \|u_1-u_2\|_V + \|p_1-p_2\|_Q \le c\,\|f_1-f_2\|_{V^*}. \]
\end{theorem}
\begin{proof}
Introduce an auxiliary functional $\Psi\colon V\to\mathbb{R}$ by
\[ \Psi(v)=I_\Delta(\psi(\gamma v)),\quad v\in V. \]
Then, by \cite[Section 3.3]{MOS2013}, $\Psi$ is well-defined and is locally Lipschitz continuous on $V$.
Moreover,
\begin{equation} 
\Psi^0(u;v)\le I_\Delta(\psi^0(\gamma u;\gamma v))\quad\forall\,u, v\in V, 
\label{Ppsi}
\end{equation}
and 
\[ \Psi^0(v_1;v_2-v_1) + \Psi^0(v_2;v_1-v_2) \le \alpha_\psi c_\Delta^2 \|v_1-v_2\|_V^2\quad\forall\,v_1,v_2\in V.\]
Hence, $H(\psi)$ implies $H(\Psi)$ with $\alpha_\Psi=\alpha_\psi c_\Delta^2$.  From Theorems \ref{thm:main2} and \ref{thm:main3}, we know that 
Problem \ref{p3} has a unique solution $(u,p)\in K_V\times K_Q$ and the solution depends Lipschitz 
continuously on $f$. Thanks to the property \eqref{Ppsi}, $(u,p)$ is also a solution of Problem \ref{p3v}. 
Moreover, the solution uniqueness of Problem \ref{p3v} can be shown by slightly modifying the argument of
the solution uniqueness of Problem \ref{p3} in the last part of the proof of Theorem \ref{thm:main2}.
Consequently, the statements of Theorem \ref{thm:main3v} hold regarding Problem \ref{p3v}.
\end{proof}

Finally, we consider degenerate special cases of Problem \ref{p3} and Problem \ref{p3v} where $\Phi$ depends on only one variable.
Let $\Phi\colon V\to\real$.

\begin{problem}\label{p1} 
Find $(u,p)\in K_V \times K_Q$ such that
\begin{align}
	&\langle Au,v-u\rangle +b(v-u,p)+\Phi(v)-\Phi(u) +\Psi^0(u;v-u)\ge\langle f,v-u\rangle
	\quad\forall\,v\in K_V,\label{eq:1}\\
	&b(u,q-p)\leq 0\quad\forall\, q\in K_Q.\label{eq:2}
\end{align}
\end{problem}

By the nomenclature used in \cite{Han2024}, Problem \ref{p1} is called a mixed variational-hemivariational 
inequality of rank $(1,1)$ since both the functions $\Phi$ and $\Psi$ depend on one variable. 
In the study of Problem \ref{p1}, we replace the assumption $H(\Phi)_2$ by $H(\Phi)_1$.

\smallskip
\begin{itemize}
\item $H(\Phi)_1$ \ $\Phi\colon V\to\mathbb{R}$ is convex and continuous, and
\begin{equation}
\Phi(v+w)=\Phi(v)\quad\forall\,v\in K_V,\,w\in V_0.
\label{Phi1}
\end{equation}
\end{itemize}
\smallskip

As a corollary of Theorems \ref{thm:main2} and \ref{thm:main3}, we have the next result on Problem \ref{p1}.

\begin{theorem}\label{thm:main1}
Assume $H(K_V)$, $H(K_Q)$, $H(A)$, $H(b)$, $H(\Phi)_1$, $H(\Psi)$, $H(f)$, and $\alpha_\Psi<m_A$.
Then, Problem \ref{p1} has a unique solution $(u,p)\in K_V\times K_Q$.  Moreover, there exists a constant $c>0$ such that
for solutions $(u_1,p_1),(u_2,p_2)\in K_V\times K_Q$ of Problem \ref{p1} with $f=f_1,f_2\in V^*$, 
the following inequality holds:
\[ \|u_1-u_2\|_V + \|p_1-p_2\|_Q \le c\,\|f_1-f_2\|_{V^*}. \]
\end{theorem}

Similar to Problem \ref{p3v}, we introduce a variant of Problem \ref{p1}.

\begin{problem}\label{p1v}
Find $(u,p)\in K_V \times K_Q$ such that
\begin{align}
&\langle Au,v-u\rangle +b(v-u,p)+\Phi(v)-\Phi(u)+I_\Delta(\psi^0(\gamma u;\gamma v-\gamma u))
\ge\langle f,v-u\rangle \quad\forall\,v\in K_V,\label{n1v1}\\
&b(u,q-p)\leq 0\quad\forall\, q\in K_Q.\label{n2v1}
\end{align}
\end{problem}

\begin{theorem}\label{thm:main1v}
Assume $H(K_V)$, $H(K_Q)$, $H(A)$, $H(b)$, $H(\Phi)_1$, $H(\psi)$, $H(f)$, and $\alpha_\psi c_\Delta^2<m_A$.
Then, Problem \ref{p1v} has a unique solution $(u,p)\in K_V\times K_Q$.  Moreover, there exists a constant $c>0$ such that
for solutions $(u_1,p_1),(u_2,p_2)\in K_V\times K_Q$ of Problem \ref{p3v} with $f=f_1,f_2\in V^*$, 
the following inequality holds:
\[ \|u_1-u_2\|_V + \|p_1-p_2\|_Q \le c\,\|f_1-f_2\|_{V^*}. \]
\end{theorem}

\section{Numerical analysis}\label{sec:na}

In this section, we consider the numerical solution of Problem \ref{p3v}. We keep the assumptions stated 
in Theorem \ref{thm:main3v} so that Problem \ref{p3v} has a unique solution.  Let $K_V^h\subset K_V$ and 
$K_Q^h\subset K_Q$ be finite dimensional approximations of $K_V$ and $K_Q$.  Then, the numerical method is
the following.

\begin{problem}\label{p3h}
Find $(u^h,p^h)\in K_V^h \times K_Q^h$ such that
\begin{align}
&\langle Au^h,v^h-u^h\rangle +b(v^h-u^h,p^h)+\Phi(u^h,v^h)-\Phi(u^h,u^h) +I_\Delta(\psi^0(\gamma u^h;\gamma v^h-\gamma u^h))\nonumber\\
&\quad \ge\langle f,v^h-u^h\rangle \quad\forall\,v^h\in K_V^h,\label{n1h}\\
&b(u^h,q^h-p^h)\leq 0\quad\forall\, q^h\in K_Q^h.\label{n2h}
\end{align}
\end{problem}

Denote 
\[ V^h_0=V_0\cap K_V^h\]
and assume the finite element space pair $V^h\times Q^h$ satisfies the discrete inf-sup condition
(Babu\v{s}ka-Brezzi condition, or simply BB condition)
\begin{equation}
\alpha_b \|q^h\|_Q\le \sup_{w^h\in V^h_0} \frac{b(w^h,q^h)}{\|w^h||_V},
\label{BB}
\end{equation}
where $\alpha_b>0$ is a constant independent of $h$.

Under the assumptions stated in Theorem \ref{thm:main3v} and \eqref{BB}, Problem \ref{p3h} has a unique solution.  The 
rest of the section is devoted to an error analysis of the numerical method.  Let $(v^h,q^h)\in K_V^h \times K_Q^h$ be arbitrary.  We will use the modified Cauchy-Schwarz inequality several times:
\begin{equation}
a\,b\le \epsilon\,a^2 + \frac{1}{4\,\epsilon}\,b^2\quad\forall\,\epsilon>0,\ a,b\in\mathbb{R}.
\label{mCS}
\end{equation}

Take $v=u^h$ in \eqref{n1} and $q=p^h$ in \eqref{n2} to get
\begin{align}
&\langle Au,u^h-u\rangle +b(u^h-u,p)+\Phi(u,u^h)-\Phi(u,u) +I_\Delta(\psi^0(\gamma u;\gamma u^h-\gamma u))\ge\langle f,u^h-u\rangle,\label{n3h}\\
&b(u,p^h-p)\leq 0.\label{n4h}
\end{align}
Write
\[ \langle Au-Au^h, u-u^h\rangle = \langle Au,u-u^h\rangle+\langle Au^h,u^h-v^h\rangle+\langle Au,v^h-u\rangle. \]
Use \eqref{n3h} and \eqref{n1h} to obtain
\begin{equation}
\langle Au-Au^h, u-u^h\rangle = I_A + I_R + I_b +I_\Phi+I_\psi,
\label{n5h}
\end{equation}
where
\begin{align}
I_A & =\langle Au-Au^h, u-v^h\rangle,\label{n6h}\\
I_R &=\langle Au,v^h-u\rangle+b(v^h-u,p)+\Phi(u,v^h)-\Phi(u,u)+I_\Delta(\psi^0(\gamma u;\gamma v^h-\gamma u))-\langle f,v^h-u\rangle,
\label{n7h}\\
I_b & = b(u^h-v^h,p-p^h),\label{n8h}\\
I_\Phi & =\Phi(u,u^h)-\Phi(u,v^h)+\Phi(u^h,v^h)-\Phi(u^h,u^h),\label{n9h}\\
I_\psi & =I_\Delta( \psi^0(\gamma u;\gamma u^h-\gamma u)+ \psi^0(\gamma u^h;\gamma u-\gamma u^h)+
\psi^0(\gamma u;\gamma u-\gamma v^h)+\psi^0(\gamma u^h;\gamma v^h-\gamma u) ). \label{n10h}
\end{align}
Note that $I_R$ is a residual-type quantity.

We now bound the terms on the right side of \eqref{n5h}.  For $I_A$, we use the Lipschitz continuity of $A$:
\[ I_A\le M_A \|u-u^h\|_V \|u-v^h\|_V. \]
For $I_\Phi$, apply \eqref{n3} followed by the use of the triangle inequality:
\[ I_\Phi \le \alpha_\Phi \|u-u^h\|_V \|u^h-v^h\|_V
\le \alpha_\Phi \|u-u^h\|_V \left(\|u-u^h\|_V + \|u-v^h\|_V\right).\]
Then, use \eqref{mCS},
\[ \alpha_\Phi \|u-u^h\|_V \|u-v^h\|_V\le \epsilon\,\|u-u^h\|_V^2 +c\, \|u-v^h\|_V^2\]
for a constant $c$ depending on $\epsilon$.  Hence,
\[ I_\Phi\le \left(\alpha_\Phi+\epsilon\right) \|u-u^h\|_V^2 + c_\epsilon \|u-v^h\|_V^2.\]
For $I_\psi$, we use $H(\psi)$ and \eqref{cd} to get
\begin{align*}
I_\psi & \le\alpha_\psi I_\Delta(|\gamma(u-u^h)|_{V_\psi}^2)+c\left(1+\|\gamma u\|_{V_\psi}+\|\gamma u^h\|_{V_\psi}\right)\|\gamma (u-v^h)\|_{V_\psi}\\
& \le\alpha_\psi c_\Delta^2 \|u-u^h\|_V^2+c\left(1+\|\gamma u\|_{V_\psi}+\|\gamma (u-u^h)\|_{V_\psi}\right)\|\gamma (u-v^h)\|_{V_\psi}\\
& \le\alpha_\psi c_\Delta^2 \|u-u^h\|_V^2+c\left(1+\|\gamma u\|_{V_\psi}\right)\|\gamma (u-v^h)\|_{V_\psi} + c\,\|u-u^h\|_V \|u-v^h\|_V.
\end{align*}
Apply the modified Cauchy-Schwarz inequality \eqref{mCS} to get
\[ c\,\|u-u^h\|_V \|u-v^h\|_V \le \epsilon\,\|u-u^h\|_V^2 + c\, \|u-v^h\|_V \]
for a constant $c$ depending on $\epsilon$ on the right side of the inequality.  Thus,
\[ I_\psi \le \left(\alpha_\psi c_\Delta^2+\epsilon\right) \|u-u^h\|_V^2+c\left(1+\|\gamma u\|_{V_\psi}\right)\|\gamma (u-v^h)\|_{V_\psi}. \]
Regarding $I_b$, write
\[ I_b =b(u^h-u,p-q^h) + b(u-v^h,p-p^h) + b(u,p-q^h)+ b(u^h,q^h-p^h) + b(u,p^h-p).\]
From \eqref{n2},
\[ b(u,p^h-p)\le 0.\]
From \eqref{n2h}, 
\[ b(u^h,q^h-p^h)\le 0. \]
Hence,
\begin{equation}
I_b \le b(u^h-u,p-q^h) + b(u-v^h,p-p^h) + b(u,p-q^h), 
\label{Ib1}
\end{equation}
from which,
\begin{equation}
I_b \le M_b \left( \|u^h-u\|_V \|p-q^h\|_Q + \|u-v^h\|_V \|p-p^h\|_Q + \|u\|_V \|p-q^h\|_Q\right).
\label{Ib2}
\end{equation}
Applying \eqref{mCS}, we then have
\[ I_b \le \epsilon\left(\|u-u^h\|_V^2 + \|p-p^h\|_Q^2\right) 
+ c\left(\|u-v^h\|_V^2 + \|p-q^h\|_Q^2 + \|p-q^h\|_Q\right). \]
Summarizing, from \eqref{n5h} and the above bounds, we derive that
\begin{align}
\left(m_A-\alpha_\Phi-3\epsilon\right)\|u-u^h\|_V^2 
& \le \epsilon\,\|p-p^h\|_Q^2 + I_R \nonumber\\
&\quad {} +c\left(\|u-v^h\|_V^2+\|\gamma(u-v^h)\|_{V_\psi} + \|p-q^h\|_Q^2 + \|p-q^h\|_Q\right). 
\label{bd:u1}
\end{align}

Turning to an error estimate on $p-p^h$, for any $q^h\in Q^h$, by the triangle inequality,
\[ \|p-p^h\|_Q \le \|p-q^h\|_Q + \|q^h-p^h\|_Q. \]
Thus,
\begin{equation}
\|p-p^h\|_Q^2 \le 2\,\|p-q^h\|_Q^2 + 2\,\|q^h-p^h\|_Q^2 .
\label{bd:p1}
\end{equation}
Apply the BB condition \eqref{BB},
\[ \alpha_b \|q^h-p^h\|_Q\le \sup_{w^h\in V^h_0} \frac{b(w^h,q^h-p^h)}{\|w^h||_V}. \]
Similar to \eqref{n23}, due to the assumptions \eqref{Phi2} and \eqref{Psi}, we have the equality
\[ \langle Au,w\rangle + b(w,p) =\langle f,w\rangle  \quad\forall\,w\in V_0.  \]
Moreover, we derive from \eqref{n1h} that
\[ \langle Au^h,w^h\rangle + b(w^h,p^h) =\langle f,w^h\rangle  \quad\forall\,w^h\in V_0^h.  \]
From the last two equalities, we find that
\[ b(w^h,p-p^h) = \langle Au^h - Au,w^h\rangle  \quad\forall\,w^h\in V_0^h.  \]
Now,
\[ b(w^h,q^h-p^h)=b(w^h,q^h-p)+b(w^h,p-p^h)=b(w^h,q^h-p)+\langle Au^h - Au,w^h\rangle.\]
Hence,
\[ \alpha_b \|q^h-p^h\|_Q\le\sup_{w^h\in V^h_0}\frac{b(w^h,q^h-p)+\langle Au^h-Au,w^h\rangle}{\|w^h||_V}
\le M_A \|u-u^h\|_V + M_b \|p-q^h\|_Q. \]
This implies 
\[ \|q^h-p^h\|_Q^2 \le \left[ (M_A/\alpha_b)^2+1\right] \|u-u^h\|_V^2 + c\, \|p-q^h\|_Q^2. \]
From the above inequality and \eqref{bd:p1}, 
\[ \|p-p^h\|_Q^2 \le 2\left[ (M_A/\alpha_b)^2+1\right] \|u-u^h\|_V^2 + c\, \|p-q^h\|_Q^2. \]
Together with \eqref{bd:u1}, when choosing $\epsilon>0$ small enough, we have two positive constants 
$c_1,c>0$ such that
\begin{equation}
\|u-u^h\|_V^2 + \|p-p^h\|_Q^2 \le c_1 I_R 
 +c\left(\|u-v^h\|_V^2 +\|\gamma(u-v^h)\|_{V_\psi} + \|p-q^h\|_Q^2 + \|p-q^h\|_Q\right). 
\label{bd:1}
\end{equation} 

Next, consider the special case where $K_Q$ and $K_Q^h$ are subspaces.  Then, \eqref{n2} is equivalent to
\[ b(u,q)=0\quad\forall\,q\in K_Q,\]
and \eqref{n2h} is equivalent to
\[ b(u^h,q^h)=0\quad\forall\,q^h\in K_Q^h.\]
Write
\[ I_b = b(u^h-u,p-p^h) + b(u-v^h,p-p^h).\]
Then,
\[ b(u^h-u,p-p^h) = b(u^h,p-p^h) = b(u^h,p-q^h) = b(u^h-u,p-q^h),\]
and so instead of \eqref{Ib1},
\begin{align*}
I_b & =b(u^h-u,p-q^h)+b(u-v^h,p-p^h)\\
& \le M_b\left( \|u-u^h\|_V \|p-q^h\|_Q + \|u-v^h\|_V \|p-p^h\|_Q\right),
\end{align*}
and instead of \eqref{Ib2}, we have
\begin{equation}
I_b \le M_b \left( \|u^h-u\|_V \|p-q^h\|_Q + \|u-v^h\|_V \|p-p^h\|_Q \right).
\label{Ib3}
\end{equation}
Eventually, \eqref{bd:1} is improved to
\begin{equation}
\|u-u^h\|_V^2 + \|p-p^h\|_Q^2\le c_1 I_R
 +c\left(\|u-v^h\|_V^2 +\|\gamma(u-v^h)\|_{V_\psi} + \|p-q^h\|_Q^2\right). 
\label{bd:2}
\end{equation}

In conclusion, we have proved the following result.

\begin{theorem}\label{thm:main3vh}
Keep the assumptions stated in Theorem \ref{thm:main3v} and \eqref{BB}.  Then, Problem \ref{p3h}
has a unique solution $(u^h,p^h)\in K_V^h \times K_Q^h$.  For any $(\bv^h,q^h)\in K_V^h\times K_Q^h$, 
we have the error bound \eqref{bd:1}.  Furthermore, if $K_Q$ and $K_Q^h$ are subspaces,
for any $(\bv^h,q^h)\in K_V^h\times K_Q^h$, we have the error bound \eqref{bd:2}.  
\end{theorem}

A crude bound on the term $I_R$ is
\[ |I_R|\le c\left(1+\|u\|_Q+\|p\|_Q+\|f\|_{V^*}\right)\|u-v^h\|_V+c\left|\Phi(u,v^h)-\Phi(u,u)\right|. \]
Then we deduce convergence of the numerical solutions from either \eqref{bd:1} or \eqref{bd:2}.

\begin{corollary}
Keep the assumptions stated in Theorem \ref{thm:main3v} and \eqref{BB}.  If $\cup_{h>0}K_V^h$ is dense 
in $K_V$ with respect to the $\|\cdot\|_V$-norm and $\cup_{h>0}K_Q^h$ is dense in $K_Q$ with respect
to the $\|\cdot\|_Q$-norm, then we have the convergence
\[ \|u-u^h\|_V + \|p-p^h\|_Q \to 0\quad\forall\,h\to 0.\]
\end{corollary}

\section{A Stokes variational-hemivariational inequality}\label{sec:Stokes}

In this section, we consider a variational-hemivariational inequality of the Stokes equations for an
incompressible fluid flow occupying the domain $\Omega$ subject to non-leaking slip boundary conditions 
of friction types on parts of the boundary.  Let $\Omega\subset\mathbb{R}^d$, $d\le 3$, be an open 
bounded Lipschitz domain.  Its boundary $\Gamma=\partial\Omega$ is split into three parts: $\Gamma_D$, 
$\Gamma_{S,1}$ and $\Gamma_{S,2}$. The unit outward normal vector exists a.e.\ on $\Gamma$ and is 
denoted by $\bnu=(\nu_1,\ldots,\nu_d)^T$.  For a vector $\bv$, its normal component and tangential
component on $\Gamma$ are $v_\nu=\bv\cdot\bnu$ and $\bv_\tau=\bv-v_\nu \bnu$.

Let $\mu>0$ be the viscosity of the fluid, and denote by $\bu$ and $p$
for the unknown velocity and pressure fields of the fluid.  Denote 
\[ \bvarepsilon(\bu)=(\nabla\bu+(\nabla\bu)^T)/2\]
for the strain rate tensor and 
\[ \bsigma=2\,\mu\,\bvarepsilon(\bu)-p\,\bI\]
for the stress tensor, $\bI$ 
being the identity tensor.  The normal component and the tangential component of the stress on the 
boundary are $\sigma_\nu=(\bsigma\bnu)\cdot\bnu$ and $\bsigma_\tau=\bsigma\bnu-\sigma_\nu\bnu$.

Given a source density function $\fb\colon \Omega\to\mathbb{R}^d$, a convex function 
$\phi\colon \Gamma_{S,1}\to\mathbb{R}$, and a non-convex function $\psi\colon \Gamma_{S,2}\to \mathbb{R}$,
the pointwise formulation of the problem is the following.

\begin{problem}\label{p:S1}
Find a velocity field $\bu\colon \Omega\to\mathbb{R}^d$ and a pressure field $p\colon \Omega\to\mathbb{R}$ such that
\begin{align}
& {}-{\rm div}(2\,\mu\,\bvarepsilon(\bu))+\nabla p=\fb\quad{\rm in}\ \Omega, \label{S1}\\
& {\rm div}\bu=0\quad{\rm in}\ \Omega, \label{S2}\\
& \bu=\bzero\quad{\rm on}\ \Gamma_D, \label{S3}\\
& u_\nu=0,\ -\bsigma_\tau\in\partial\phi(\bu_\tau)\quad{\rm on}\ \Gamma_{S,1}, \label{S4}\\ 
& u_\nu=0,\ -\bsigma_\tau\in\partial\psi(\bu_\tau)\quad{\rm on}\ \Gamma_{S,2}. \label{S5}
\end{align}
\end{problem}

Problem \ref{p:S1} will be studied through its formulation.  For this purpose, we introduce the following function spaces:
\begin{align}
& V=\left\{\bv\in H^1(\Omega)^d\mid \bv=\bzero\ {\rm on}\ \Gamma_D,\, v_\nu=0\ {\rm on}\ \Gamma_S\right\},
\label{Sp:V}\\
& V_{\rm div}=\left\{\bv\in V\mid {\rm div}\bv=0\ {\rm in}\ \Omega\right\},\label{Sp:Vdiv}\\
& V_0=H^1_0(\Omega)^d,\label{Sp:V0}\\
& Q=\left\{q\in L^2(\Omega)\mid (q,1)_\Omega=0\right\}, \label{Sp:Q}
\end{align}
where $\Gamma_S=\Gamma_{S,1}\cup\Gamma_{S,2}$.  
Over the space $V$, we use the norm $\|\bv\|_V=\|\bvarepsilon(\bv)\|_{L^2(\Omega)^{d\times d}}$, 
which is equivalent to the standard $H^1(\Omega)^d$ norm over $V$.

The reduced weak formulation of Problem \ref{p:S1} is the following.

\begin{problem}\label{p:S2}
Find $\bu\in V_{\rm div}$ such that
\begin{align}
&\int_\Omega 2\,\mu\,\bvarepsilon(\bu):\bvarepsilon(\bv-\bu)\,dx
 +\int_{\Gamma_{S,1}}\left[\phi(\bv_\tau)-\phi(\bu_\tau)\right] ds
 +\int_{\Gamma_{S,2}} \psi^0(\bu_\tau;\bv_\tau-\bu_\tau)\,ds \nonumber\\
& {}\qquad \ge (\fb,\bv-\bu)\quad \forall\,\bv\in V_{\rm div}.
\label{S10}
\end{align} 
\end{problem}

The full weak formulation of Problem \ref{p:S1} that includes the pressure unknown $p$, is the following.

\begin{problem}\label{p:S3}
Find $(\bu,p)\in V\times Q$ such that
\begin{align}
&\int_\Omega 2\,\mu\,\bvarepsilon(\bu):\bvarepsilon(\bv-\bu)\,dx -\int_\Omega p\,{\rm div}(\bv-\bu)\,dx
+ \int_{\Gamma_{S,1}}\left[\phi(\bv_\tau)-\phi(\bu_\tau)\right] ds \nonumber\\
&{}\qquad +\int_{\Gamma_{S,2}} \psi^0(\bu_\tau;\bv_\tau-\bu_\tau)\,ds\ge (\fb,\bv-\bu)\quad \forall\,\bv\in V,
\label{S11}\\
& \int_\Omega q\,{\rm div}\bu\,dx=0\quad\forall\,q\in Q.
\label{S12}
\end{align}
\end{problem}

We will apply the theoretical results in the last two sections to study Problem \ref{p:S3} and its 
numerical approximation.  Let $K_V=V$, $K_Q=Q$, and define for $\bu,\bv\in V$ and $q\in Q$,
\begin{align*}
& a(\bu,\bv) = \int_\Omega 2\,\mu\,\bvarepsilon(\bu):\bvarepsilon(\bv)\,dx,\\
& b(\bv,q) = -\int_\Omega q\,{\rm div}\bv\,dx,\\
& \Phi(\bv) = \int_{\Gamma_{S,1}}\phi(\bv_\tau)\,ds,\\
& (\fb,\bv) = \int_\Omega \fb\cdot\bv\,dx.
\end{align*}
It is well known that the inf-sup condition \eqref{infsup} is satisfied.

For the functions $\phi$ and $\psi$, we introduce the following assumptions.

\begin{itemize}
\item $H(\phi)$ \ $\phi\colon \mathbb{R}^d\to\mathbb{R}$ is convex and Lipschitz continuous.
\item $H(\psi)$ \ $\psi\colon \mathbb{R}^d\to\mathbb{R}$ is locally Lipschitz continuous and for some 
constants $c_\psi\ge 0$ and $\alpha_\psi\ge 0$,
\begin{align}
& |\partial \psi(\bz)|_{\mathbb{R}^d} \le c_\psi\left(1+|\bz|_{\mathbb{R}^d} \right)\quad\forall\,\bz\in\mathbb{R}^d, \label{psi3}\\
& \psi^0(\bz_1;\bz_2-\bz_1) + \psi^0(\bz_2;\bz_1-\bz_2)\le \alpha_\psi |\bz_1-\bz_2|^2_{\mathbb{R}^d}
\quad\forall\,\bz_1,\bz_2\in\mathbb{R}^d. \label{psi4}
\end{align}
\end{itemize}

\begin{theorem}\label{thm:ex1}
Assume $H(\phi)$, $H(\psi)$, and $\alpha_\psi<2\,\mu\,\lambda_0$, where $\lambda_0>0$ is the smallest eigenvalue of the eigenvalue problem
\[ \bu\in V,\quad \int_\Omega 2\,\mu\,\bvarepsilon(\bu):\bvarepsilon(\bv)\,dx=\lambda \int_{\Gamma_{S,2}}\bu_\tau{\cdot}\bv_\tau \,ds \quad \forall\,\bv\in V. \]
Then for any $\fb\in L^2(\Omega)^d$, both Problem \ref{p:S2} and Problem \ref{p:S3} admit a unique solution and the two problems are 
equivalent in the sense that if $\bu$ is a solution of Problem \ref{p:S2}, then there exists $p$ such that
$(\bu,p)$ solves Problem \ref{p:S3}, and conversely, if $(\bu,p)$ is a solution of Problem \ref{p:S3}, then $\bu$ solves Problem \ref{p:S2}.
\end{theorem}
\begin{proof}
First, we apply Theorem \ref{thm:main3v} to know that Problem \ref{p:S3} has a unique solution.  Note that 
$c_\Delta^2=\lambda_0^{-1}$.   
By \cite[Theorem 5.9]{Han2024}, we know Problem \ref{p:S2} has a unique solution.

Suppose $(\bu,p)\in V\times Q$ is the unique solution of Problem \ref{p:S3}.  Then we take $q={\rm div}\bu\in Q$ in \eqref{S12} to see that $\bu\in V_{\rm div}$.  By restricting $\bv$ to $V_{\rm div}$, we derive \eqref{S10} from \eqref{S11}.

Conversely, suppose $\bu$ is the unique solution of Problem \ref{p:S2}.  Then from the previous paragraph, 
$\bu$ is the first component of the unique solution $(\bu,p)\in V\times Q$ to Problem \ref{p:S3}.
\end{proof}

Turn now to discussion of a finite element approximation of Problem \ref{p:S3}.  For simplicity, assume
$\Omega$ is a polygonal/polyhedral domain.  Consider a regular family of finite element partitions 
$\{{\cal T}^h\}_h$ of $\overline{\Omega}$ into triangular elements (if $d=2$) or tetrahedral elements
(if $d=3$).  We assume that each partition is compatible to the boundary splitting according to the
boundary condition types, i.e., if a side or a face of an element has a nontrivial intersection 
with $\Gamma_D$ or $\Gamma_{S,1}$ or $\Gamma_{S,2}$, then the side or the face lies in 
$\overline{\Gamma_D}$ or $\overline{\Gamma_{S,1}}$ or $\overline{\Gamma_{S,2}}$.  Let $V^h\times Q^h
\subset V\times Q$ be a stable pair of finite element spaces corresponding to the partition ${\cal T}^h$,
i.e., for a constant $\beta_b>0$, independent of $h$, such that
\begin{equation}
\beta_b \|q^h\|_Q\le\sup_{{\boldsymbol v}^h\in V_0^h} \frac{b(\bv^h,q^h)}{\|\bv^h\|_V}
\quad\forall\,q^h\in Q^h,
\label{infsuph}
\end{equation}
where 
\[ V^h_0 = V^h \cap V_0. \]
Then the finite element method for solving Problem \ref{p:S3} is the following.

\begin{problem}\label{p:S3h}
Find $(\bu^h,p^h)\in V^h\times Q^h$ such that
\begin{align}
&\int_\Omega 2\,\mu\,\bvarepsilon(\bu^h):\bvarepsilon(\bv^h-\bu^h)\,dx -\int_\Omega p^h\,{\rm div}(\bv^h-\bu^h)\,dx
+ \int_{\Gamma_{S,1}}\left[\phi(\bv^h_\tau)-\phi(\bu^h_\tau)\right] ds \nonumber\\
&{}\qquad +\int_{\Gamma_{S,2}} \psi^0(\bu^h_\tau;\bv^h_\tau-\bu^h_\tau)\,ds\ge (\fb,\bv^h-\bu^h)\quad \forall\,\bv^h\in V^h,
\label{S11h}\\
& \int_\Omega q^h\,{\rm div}\bu^h\,dx=0\quad\forall\,q^h\in Q^h.
\label{S12h}
\end{align}
\end{problem}

Under the assumptions stated in Theorem \ref{thm:ex1}, we can apply Theorem \ref{thm:main1v} to 
conclude that Problem \ref{p:S3h} has a unique solution.  The rest of the section focuses on an 
error analysis of the numerical method.  By \eqref{bd:2},
\begin{equation}
\|\bu-\bu^h\|_V^2 + \|p-p^h\|_Q^2\le c_1 I_R
 +c\left(\|\bu-\bv^h\|_V^2 +\|\bu_\tau-\bv^h_\tau\|_{L^2(\Gamma_S)^2} + \|p-q^h\|_Q^2\right),
\label{ex1:Cea1}
\end{equation}
where from \eqref{n7h},
\[ I_R = a(\bu,\bv^h-\bu)+b(\bv^h-\bu,p)+I_{\Gamma_{S,1}}(\phi(\bv^h_\tau)-\phi(\bu_\tau))
+I_{\Gamma_{S,2}}(\psi^0(\bu_\tau;\bv^h_\tau-\bu_\tau))-(\fb,\bv^h-\bu). \]

Let $\{\Gamma_i\}_{1\le i\le i_0}$ be the plat components of $\Gamma_S=\Gamma_{S,1}\cup\Gamma_{S,2}$.
Assume the solution regularities:
\begin{equation}
\bu\in H^2(\Omega)^d, \quad \bu_\tau|_{\Gamma_i}\in H^2(\Gamma_i)^d,\ 1\le i\le i_0,
\quad p\in H^1(\Omega).
\label{reg1}
\end{equation}
As in the proof of Theorem 4.5 in \cite{FCHCD20}, it can be shown that the solution $(\bu,p)\in V\times Q$ of Problem \ref{p:S3} satisfies the 
following pointwise relations:
\begin{align}
& {}-{\rm div}(2\,\mu\,\bvarepsilon(\bu))+\nabla p=\fb\quad{\rm a.e.\ in}\ \Omega, \label{S1p}\\
& {\rm div}\bu=0\quad{\rm a.e.\ in}\ \Omega. \label{S2p}
\end{align}
Also, note the boundary values from $\bu\in V$ and the definition of the space $V$:
\begin{align}
& \bu=\bzero\quad{\rm a.e.\ on}\ \Gamma_D, \label{S3p}\\
& u_\nu=0\quad{\rm a.e.\ on}\ \Gamma_{S,1}\cup\Gamma_{S,2}. \label{S5p}
\end{align}
Performing an integration by parts and using the above pointwise relations, we have 
\begin{align*}
I_R & = \int_\Gamma \bsigma\bnu\cdot(\bv^h-\bu)\,ds
+\int_\Omega\left[ -{\rm div}(2\,\mu\,\bvarepsilon(\bu))+\nabla p-\fb\right] dx\\
&\quad{} +I_{\Gamma_{S,1}}(\phi(\bv^h_\tau)-\phi(\bu_\tau))
+I_{\Gamma_{S,2}}(\psi^0(\bu_\tau;\bv^h_\tau-\bu_\tau))\\
&= \int_{\Gamma_S} \bsigma_\tau\cdot(\bv^h_\tau-\bu_\tau)\,ds
+I_{\Gamma_{S,1}}(\phi(\bv^h_\tau)-\phi(\bu_\tau))
+I_{\Gamma_{S,2}}(\psi^0(\bu_\tau;\bv^h_\tau-\bu_\tau)).
\end{align*}
Then,
\[ I_R \le c\,\|\bu_\tau-\bv^h_\tau\|_{L^2(\Gamma_S)^d}.\]
In conclusion, we derive from \eqref{ex1:Cea1} that
\begin{equation}
\|\bu-\bu^h\|_V^2 + \|p-p^h\|_Q^2\le 
c\left(\|\bu-\bv^h\|_V^2 +\|\bu_\tau-\bv^h_\tau\|_{L^2(\Gamma_S)^2} + \|p-q^h\|_Q^2\right).
\label{ex1:Cea2}
\end{equation}

As an example, if we use the P1b/P1 element pair (\cite{ABF84}), by using the standard finite element 
approximation theory (cf.\ \cite{BS2008, Ci1978, EG2021I}), under the solution regularity assumptions \eqref{reg1},
we obtain from \eqref{ex1:Cea2} the optimal-order error estimate
\begin{equation}
\|\bu-\bu^h\|_V^2 + \|p-p^h\|_Q^2\le c\,h^2.
\label{ex1:Cea3}
\end{equation}

\section{Numerical examples}\label{sec:ex}

In the numerical examples, we let
$$\phi(\bv_\tau) = g \, |\bv_\tau|,$$
where $g$ is a positive constant, and let
\[\psi(\bz)=\int_{0}^{|{\boldsymbol z}|} \omega(t)\, dt,\  \bz \in \mathbb{R}^{d}, \quad \omega(t)=(a-b) e^{-\beta t}+b,\]
where $a, b, \beta$ are all positive constants, and $a>b$. Since $\omega$ is decreasing on $[0,+\infty), \psi$ is
not convex. The slip boundary condition on $\Gamma_{S,2}$ is equivalent to
\[\left|\bsigma_{\tau}\right| \leq \omega(0) \text { if } \bu_{\tau}=\bzero, \quad
-\bsigma_{\tau}=\omega\left(\left|\bu_{\tau}\right|\right) \frac{\bu_{\tau}}{\left|\bu_{\tau}\right|} \text { if } \bu_{\tau} \neq \bzero .
\]

Define two Lagrangian multipliers $\blambda_1 = -\bsigma_{\tau} / g$, 
$\blambda_2=-\bsigma_{\tau} / \omega\left(\left|\bu_{\tau}\right|\right)$, elements of the sets
\begin{align*}
\Lambda_1 & =\left\{\blambda_1 \in L^2\left(\Gamma_{S,1}; \mathbb{R}^{d}\right)\mid |\blambda_1| \le 1 \ {\rm a.e.\ on}\ \Gamma_{S,1}\right\},\\
\Lambda_2 & =\left\{\blambda_2\in L^2\left(\Gamma_{S,2}; \mathbb{R}^{d}\right)\mid |\blambda_2|\leq 1 \
{\rm a.e.\  on}\ \Gamma_{S,2}\right\}, 
\end{align*}
respectively, and introduce the following form of the weak formulation of the problem.

\begin{problem}\label{prob6.1}
Find $(\bu, p, \blambda_1, \blambda_2) \in \mathcal{W} \times \mathcal{Q}\times \Lambda_1\times \Lambda_2$ such that
\begin{align*}
& a(\bu, \bv)-b(\bv, p)+\int_{\Gamma_{S,1}} g\,\blambda_1 \cdot \bv_\tau ds
+\int_{\Gamma_{S,2}} \omega\left(\left|\bu_{\tau}\right|\right) \blambda_2 \cdot \bv_{\tau} ds
=\langle\fb, \bv\rangle \quad\forall\,\bv \in \bV, \\
& b(\bu, q)=0 \quad\forall\, q \in Q , \\
&\blambda_1 \cdot \bu_{\tau}=\left|\bu_{\tau}\right|\quad {\rm a.e.\ on}\ \Gamma_{S,1},\\
&\blambda_2 \cdot \bu_{\tau}=\left|\bu_{\tau}\right|\quad {\rm a.e.\ on}\ \Gamma_{S,2}.\\
\end{align*}
\end{problem}

We use the P1b/P1 finite element pair for the spatial discretization. It is known that the discrete 
inf-sup condition is satisfied (\cite{ABF84}).  Let $\boldsymbol{P}: \mathbb{R}^{d} \rightarrow
\overline{B(\bzero;1)}$ be the orthogonal projection operator onto the closed unit ball in $\mathbb{R}^d$.
We implement the following Uzawa algorithm for the numerical solution of Problem \ref{prob6.1}.

\emph{Initialization}. Choose a parameter $\rho>0$, a maximal iteration number $l_{0}$ and error tolerance $\varepsilon>0$, 
let the approximate initial velocity $\bu_0^h=\bzero$, and set $\blambda_{1,0}^h=\blambda_{2,0}^h=\bzero$.

\emph{Iteration}. For $l \geq 1$, find $\left(\bu_l^h,p_l^h\right) \in \bV^{h} \times Q^{h}$ such that
\begin{align*}
& a\left(\bu_{l}^{h}, \bv^{h}\right)-b\left(\bv^{h}, p_{l}^{h}\right) \\
& \qquad=\left\langle\fb, \bv^{h}\right\rangle-\int_{\Gamma_{S,1}} g\,\blambda_{1,l-1}^h \cdot \bv^h_\tau ds
-\int_{\Gamma_{S,2}} \omega\left(\left|\bu_{l-1,\tau}^{h}\right|\right)\blambda_{2,l-1}^{h}\cdot\bv_{\tau}^{h} ds 
\quad \forall\,\bv^{h} \in \bV^{h}, \\
& b\left(\bu_{l}^{h}, q^{h}\right)=0 \quad \forall\, q^{h} \in Q^{h}
\end{align*}
and update
\begin{align*}
&\blambda_{1,l}^{h}=\boldsymbol{P}\left(\blambda_{1,l-1}^{h}+\rho\bu_{l,\tau}^{h}\right)\quad{\rm on}\ \Gamma_{S,1},\\
&\blambda_{2,l}^{h}=\boldsymbol{P}\left(\blambda_{2,l-1}^{h}+\rho\bu_{l,\tau}^{h}\right)\quad{\rm on}\ \Gamma_{S,2}
\end{align*}
until $l=l_0$ or $\left\|\bu_l^h-\bu_{l-1}^h\right\|_{L^2(\Omega)^d}/\left\|\bu_l^h\right\|_{L^2(\Omega)^d}<\varepsilon$.

The algorithm is implemented in Python with the FEniCS package (\cite{LL2016}). In the numerical examples, we let $\mu=1$ for the viscosity coefficient,
$d=2$ for the spatial dimension, $\rho=10$. For the function $\omega(t)$, use $\beta=1$, $a=0.35$, $b=0.25$. 
For the function $\phi(t)$, use $g=0.2$. The source function is defined by
\begin{align*}
    \fb(x,y) = 2\left(
    \begin{aligned}
    -10 (1 - 6x + 6x^2) y(1 - y)(1 - 2y) + 30x^2(1 - x)^2(1 - 2y) - (1 - 2y) \\
    10 (1 - 6y + 6y^2) x(1 - x)(1 - 2x) - 30y^2(1 - y)^2(1 - 2x) - (1 - 2x)
    \end{aligned}
    \right).
\end{align*}
For the stopping criteria, let $l_{0}=200$ and $\varepsilon=10^{-6}$. 

\medskip
\noindent{\bf Example \ref{sec:ex}.1}.  In the first example, we use uniform triangulations of the 
domain $\Omega=(0,1) \times(0,1)$. The unit interval $[0,1]$ is split into $1/h$ ($h=1/8$, $1/16$, $\cdots$) equal length sub-intervals. 
The triangulation of the unit square with $h = 1/16$ is shown in Figure~\ref{Mesh_EX1}. Its boundary 
is split into three parts $\Gamma_D$, $\Gamma_{S,1}$, and $\Gamma_{S,2}$; $\Gamma_D$ is the union of the left
and right sides, $\Gamma_{S,1}$ is the top side, and $\Gamma_{S,2}$ is the bottom side.

\begin{figure}
    \centering
    \includegraphics[width=0.5\linewidth]{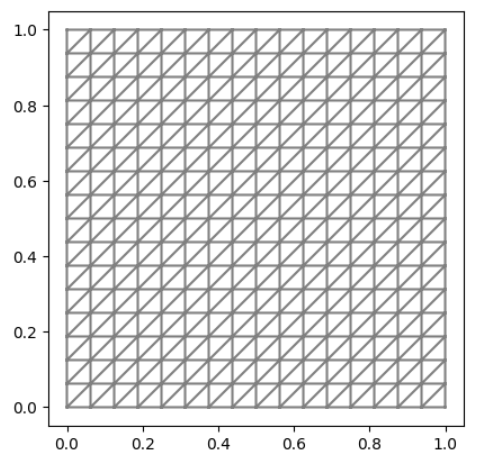}
    \caption{Triangulation for Example \ref{sec:ex}.1}
    \label{Mesh_EX1}
\end{figure}

To calculate numerical errors, we use the numerical solution with $h=1/256$ as the reference solution 
$(\bu^*,\blambda_1^*,\blambda_2^*)$. We report numerical convergence orders of
$\bu^{h}-\bu^*$ in the $(L^2(\Omega))^2$-norm and the $(H^1(\Omega))^2$-norm, that of $p^{h}-p^*$ in the 
$L^2(\Omega)$-norm, that of $\blambda_1^{h}-\blambda_1^*$ in the $(L^2(\Gamma_{S,1}))^2$-norm, and that of 
$\blambda_2^{h}-\blambda_2^*$ in the $(L^2(\Gamma_{S,2}))^2$-norm.

Tables \ref{EX1_table1} and \ref{EX1_table2} show numerical convergence orders with respect to $h$. We also report the number of iterations required for the Uzawa algorithm.
Values of the $(L^2(\Omega))^2$-norm and $(H^1(\Omega))^2$-norm of $\bu^*$, the $L^2(\Omega)$-norm of $p^*$, 
the $(L^2(\Gamma_{S,1}))^2$-norm of $\blambda_1^*$, and the $(L^2(\Gamma_{S,2}))^2$-norm of $\blambda_2^*$ 
for the reference solution are
\begin{align*}
& \| \bu^* \|_{L^2(\Omega)^2} \doteq 4.619\times 10^{-2},\quad
   \| \bu^* \|_{H^1(\Omega)^2}  \doteq 3.047\times 10^{-1},\quad
   \|  p^* \|_{L^2(\Omega)}  \doteq 2.340\times 10^{-1},\\
&   \| \blambda_1^* \|_{L^2(\Gamma_{S,1})^2} \doteq 9.261\times 10^{-1},\quad
   \| \blambda_2^* \|_{L^2(\Gamma_{S,2})^2} \doteq 8.533\times 10^{-1}.
\end{align*}

Numerical results for $h=1/16$ are shown in Figures~\ref{u_E1}, \ref{p_E1} and \ref{VelocityMag_E1}. 
To illustrate the friction effect, in Figures \ref{u_tau1_E1}, \ref{sigma_tau1_E1}, \ref{u_tau2_E1}
and \ref{sigma_tau2_E1}, we show $\bu_\tau$ and $\bsigma_\tau$ along the slip boundary $\Gamma_{S,1}$, and 
$\Gamma_{S,2}$ for the reference solution, respectively. The top and bottom sides of the domain are straight 
and parallel to the $x$-axis. Therefore, the tangential velocity $\bu_\tau$ on the side can be equated with
its first component, and similarly for the friction force $\bsigma_\tau$.

\begin{table}[h!]
\centering
   \caption{Numerical convergence orders of $\bu$ and $p$ for Example \ref{sec:ex}.1}
\begin{tabular}{c c c c c c c} 
\toprule
 $h$ & $\|\bu^h-\bu^*\|_{L^2}$ & Order & $\|\bu^h-\bu^*\|_{H^1}$ & Order & $\|p^h-p^*\|_{L^2}$ & Order\\ 
\midrule
 1/8  & 6.505e-03 & -    & 9.472e-02 & -    & 7.661e-02 & -    \\ 
 1/16 & 1.860e-03 & 1.81 & 4.750e-02 & 1.00 & 2.774e-02 & 1.47 \\
 1/32 & 4.877e-04 & 1.93 & 2.342e-02 & 1.02 & 8.712e-03 & 1.67 \\
 1/64 & 1.191e-04 & 2.03 & 1.149e-02 & 1.03 & 2.568e-03 & 1.76 \\
\midrule
\end{tabular}
\label{EX1_table1}
\end{table}

\begin{table}[h!]
\centering
   \caption{Numerical convergence orders of $\blambda_1$ and $\blambda_2$ for Example \ref{sec:ex}.1}
\begin{tabular}{c c c c c c} 
\toprule
$h$&$\|\blambda_1^h-\blambda_1^*\|_{L^2(\Gamma_{S,1})}$&Order&$\|\blambda_1^h-\blambda_1^*\|_{L^2(\Gamma_{S,1})}$ & Order & Iterations\\ 
\midrule
 1/8  & 3.395e-01 & -    & 3.320e-01 & -    & 89\\ 
 1/16 & 2.125e-01 & 0.68 & 1.313e-01 & 1.34 & 149\\
 1/32 & 7.304e-02 & 1.54 & 4.589e-02 & 1.52 & 144\\
 1/64 & 2.489e-02 & 1.55 & 1.500e-02 & 1.61 & 159\\
\midrule
\end{tabular}
\label{EX1_table2}
\end{table}

\begin{figure}[h!]
   \begin{minipage}{0.48\textwidth}
     \centering
     \includegraphics[width=.7\linewidth]{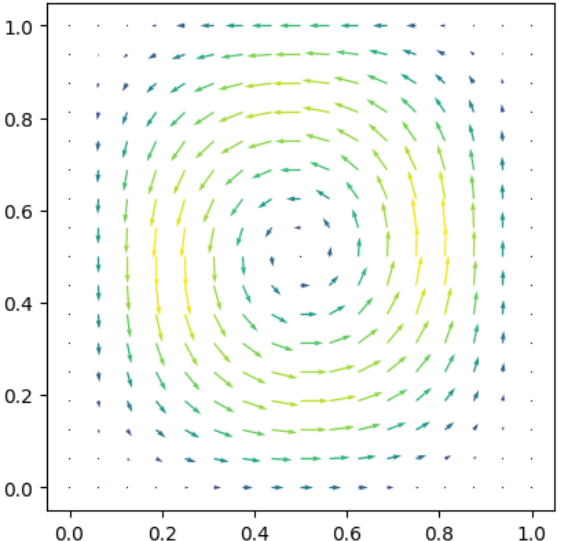}
     \caption{Velocity field for Example \ref{sec:ex}.1}
     \label{u_E1}
   \end{minipage}\hfill
   \begin{minipage}{0.48\textwidth}
     \centering
     \includegraphics[width=.7\linewidth]{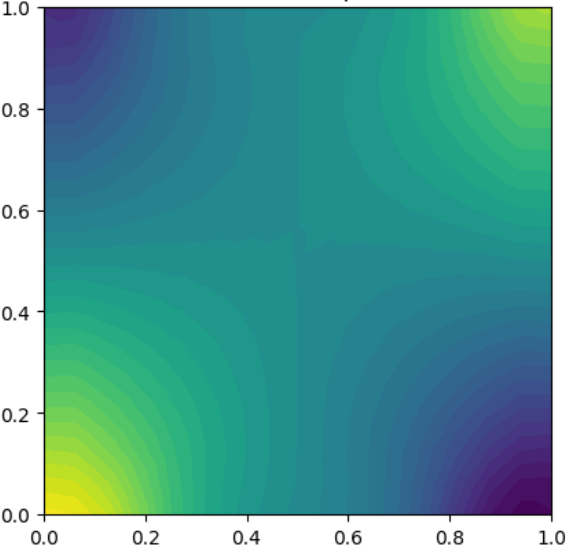}
     \caption{Pressure isobars for Example \ref{sec:ex}.1}
     \label{p_E1}
   \end{minipage}
\end{figure}

\begin{figure}
    \centering
    \includegraphics[width=0.5\linewidth]{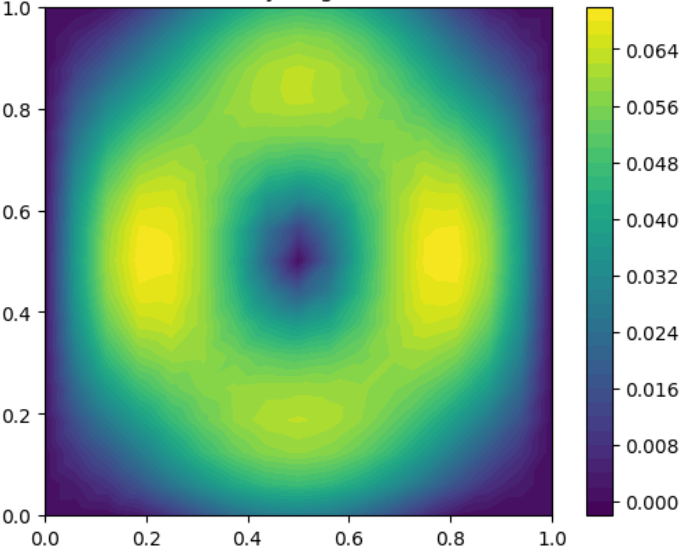}
    \caption{Magnitude of the velocity $|\bu^{h}|$ for Example \ref{sec:ex}.1}
    \label{VelocityMag_E1}
\end{figure}

\begin{figure}[h!]
   \begin{minipage}{0.48\textwidth}
     \centering
     \includegraphics[width=.7\linewidth]{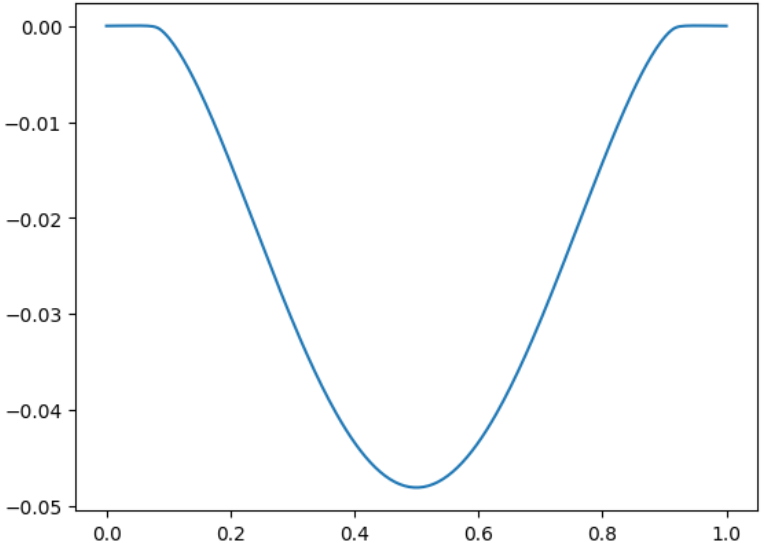}
     \caption{$ \bu_{\tau}$ along $\Gamma_{S,1}$, Example \ref{sec:ex}.1}
     \label{u_tau1_E1}
   \end{minipage}\hfill
   \begin{minipage}{0.48\textwidth}
     \centering
     \includegraphics[width=.7\linewidth]{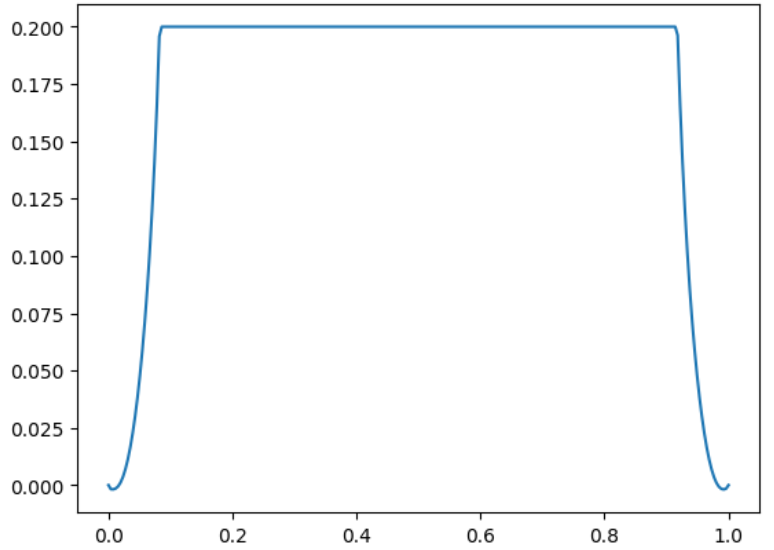}
     \caption{$ \bsigma_{\tau} $ along $\Gamma_{S,1}$, Example \ref{sec:ex}.1}
     \label{sigma_tau1_E1}
   \end{minipage}
\end{figure}

\begin{figure}[h!]
   \begin{minipage}{0.48\textwidth}
     \centering
     \includegraphics[width=.7\linewidth]{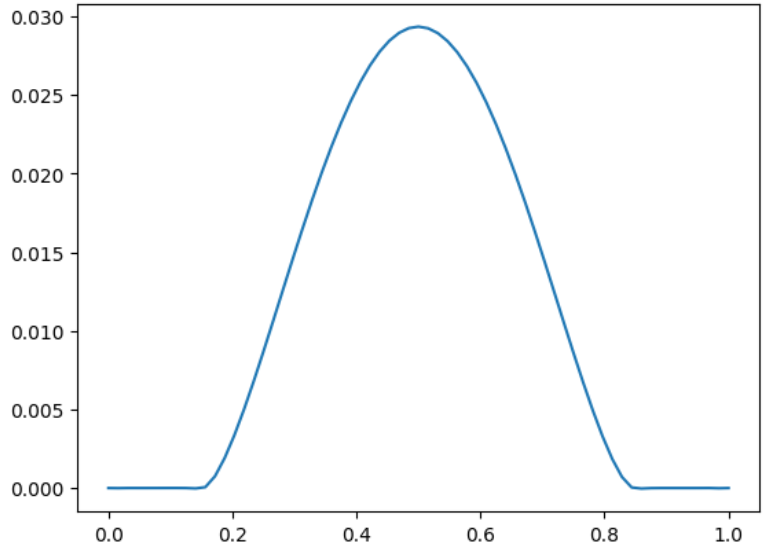}
     \caption{$ \bu_{\tau} $ along $\Gamma_{S,2}$, Example \ref{sec:ex}.1}
     \label{u_tau2_E1}
   \end{minipage}\hfill
   \begin{minipage}{0.48\textwidth}
     \centering
     \includegraphics[width=.7\linewidth]{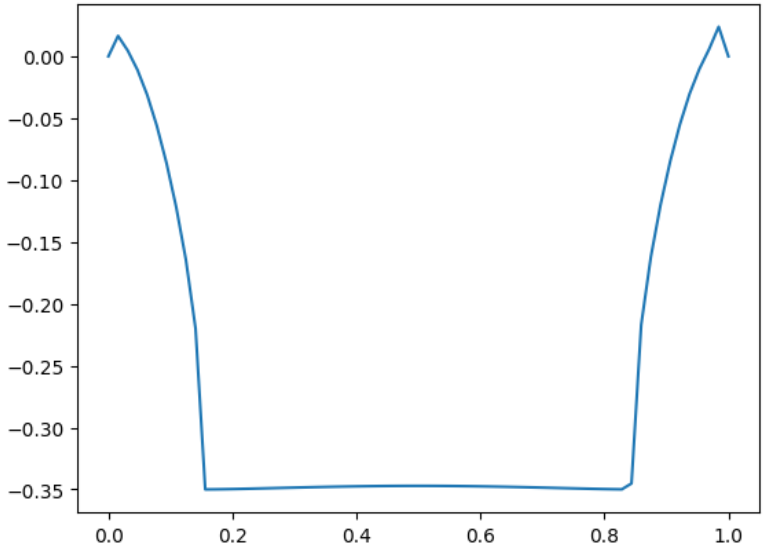}
     \caption{$ \bsigma_{\tau} $ along $\Gamma_{S,2}$, Example \ref{sec:ex}.1}
     \label{sigma_tau2_E1}
   \end{minipage}
\end{figure}

\smallskip
\noindent{\bf Example \ref{sec:ex}.2}. In the second example, the domain is a curved variant of the unit square, with the bottom side replaced by a curve defined by 
$$  y = \frac{1}{10} \sin (2 \pi x), \quad x \in [0,1],$$
We remove a disk of radius $0.2$ centered at $(0.5,0.6)^T$ from the domain and choose the circle to be $\Gamma_{S,1}$. 
Moreover, $\Gamma_D$ is the union of the top, left, and right sides, and $\Gamma_{S,2}$ is the bottom side.
We create the mesh by applying Gmsh, which is a 3D finite element mesh generator, cf.\ the website {\tt https://gmsh.info/} developed and manged by C. Geuzaine and J.-F. Remacle. In Gmsh, let the element size factor be 
$0.4$, $0.26$, $0.181$, $0.1264$, $0.06285$, and set the side length $h = 1/8$, $1/16$, $1/32$, $1/64$, $1/256$ 
on the left, top, and right boundaries. As a result, we get the total number of cells $N_{cells} = 128$, $511$, $2055$, $8187$, $131179$, respectively. The triangulations of the regions with $h = 1/16$ are shown in Figure~\ref{Mesh_EX2}.

\begin{figure}
    \centering
    \includegraphics[width=0.5\linewidth]{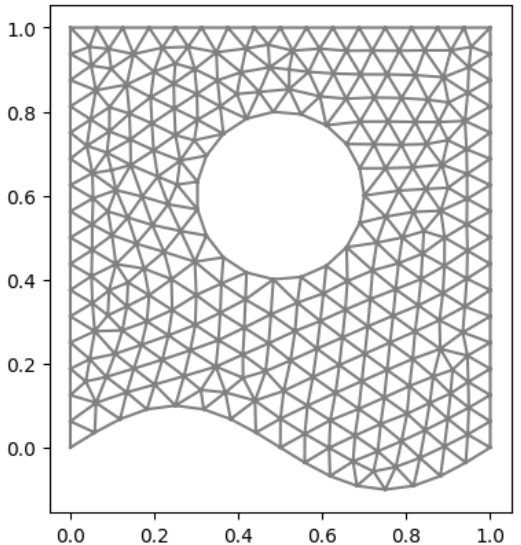}
    \caption{Triangulation for Example \ref{sec:ex}.2}
    \label{Mesh_EX2}
\end{figure}

As in Example \ref{sec:ex}.1, we use the numerical solution with $h=1/256$ as the reference solution 
$(\bu^*,\blambda_1^*,\blambda_2^*)$ in computing errors of the numerical solutions on coarse meshes. For the function 
$\omega(t)$, use $\beta=1$, $a=0.2$, $b=0.1$. Table~\ref{EX2_table1} and Table~\ref{EX2_table2} show the numerical 
convergence orders as $h$ decreases. Values of the $(L^2(\Omega))^2$-norm and $(H^1(\Omega))^2$-norm of $\bu^*$, 
the $L^2(\Omega)$-norm of $p^*$, and the $(L^2(\Gamma_S))^2$-norm of $\blambda_1^*$ and $\blambda_2^*$ for 
the reference solution:
\begin{align*}
&  \| \bu^* \|_{L^2(\Omega)^2} \doteq 2.216\times 10^{-2},\quad
   \| \bu^* \|_{H^1(\Omega)^2} \doteq 1.914\times 10^{-1},\quad
   \|  p^* \|_{L^2(\Omega)} \doteq 5.499\times 10^{-1},\\
&   \| \blambda_1^* \|_{L^2(\Gamma_{S,1})^2} \doteq 1.104,\quad
   \| \blambda_2^* \|_{L^2(\Gamma_{S,2})^2} \doteq 7.975\times 10^{-1}.
\end{align*}

\begin{table}[h!]
\centering
   \caption{Numerical convergence orders of $\bu$ and $p$ for Example \ref{sec:ex}.2}
\begin{tabular}{c c c c c c c} 
\toprule
 $h$ & $\| \bu^h - \bu^* \|_{L^2}$ & Order & $\| \bu^h - \bu^* \|_{H^1}$ & Order & $\| p^h - p^* \|_{L^2}$ & Order\\ 
\midrule
 1/8  & 5.637e-03 & -    & 8.980e-02 & -    & 1.536e-01 & -    \\ 
 1/16 & 2.053e-03 & 1.46 & 4.100e-02 & 1.13 & 5.337e-02 & 1.53 \\
 1/32 & 8.234e-04 & 1.32 & 2.017e-02 & 1.02 & 2.085e-02 & 1.36 \\
 1/64 & 3.236e-04 & 1.35 & 9.802e-03 & 1.04 & 8.486e-03 & 1.30 \\
\midrule
\end{tabular}
\label{EX2_table1}
\end{table}

\begin{table}[h!]
\centering
   \caption{Numerical convergence orders of $\blambda_1$ and $\blambda_2$ for Example \ref{sec:ex}.2}
\begin{tabular}{c c c c c c} 
\toprule
 $h$  & $\| \blambda_1^h - \blambda_1^* \|_{L^2(\Gamma_{S,1})}$ & Order &  $\| \blambda_1^h - \blambda_1^* \|_{L^
2(\Gamma_{S,1})}$ & Order & Iterations\\ 
\midrule
 1/8  & 1.090e-01 & -    & 6.712e-01 & -    & 145\\ 
 1/16 & 2.667e-02 & 2.03 & 4.032e-01 & 0.74 & 184\\
 1/32 & 1.732e-02 & 0.62 & 2.058e-01 & 0.97 & 205\\
 1/64 & 8.235e-02 & 1.07 & 1.005e-01 & 1.03 & 202\\
\midrule
\end{tabular}
\label{EX2_table2}
\end{table}

The numerical results when $h=1/16$ are shown in Figures \ref{u_E2}, \ref{p_E2}, and \ref{VelocityMag_E2}.
The magnitudes of the tangential velocity and the friction force on the slip boundary for the reference solution 
are shown in Figures \ref{u_tau1_E2}, \ref{sigma_tau1_E2}, \ref{u_tau2_E2}, and \ref{sigma_tau2_E2}.
In Figure \ref{u_tau1_E2}, the upper curve corresponds to $|\bu_\tau|$ on the top semicircular hole, while the lower curve corresponds to $|\bu_\tau|$ on the bottom semicircular hole. From Figures \ref{u_tau1_E2} and \ref{sigma_tau1_E2}, we observe that on the lower portion of the hole, for $x\in(0.495076, 0.600568)$, 
we have $|\bu_\tau| = 0$, corresponding to $|\bsigma_\tau| < 0.2$. On the remaining part of the hole, 
we have $|\bsigma_\tau| = 0.2$ and consequently $|\bu_\tau| \ne 0$.

\begin{figure}[h!]
   \begin{minipage}{0.48\textwidth}
     \centering
     \includegraphics[width=.7\linewidth]{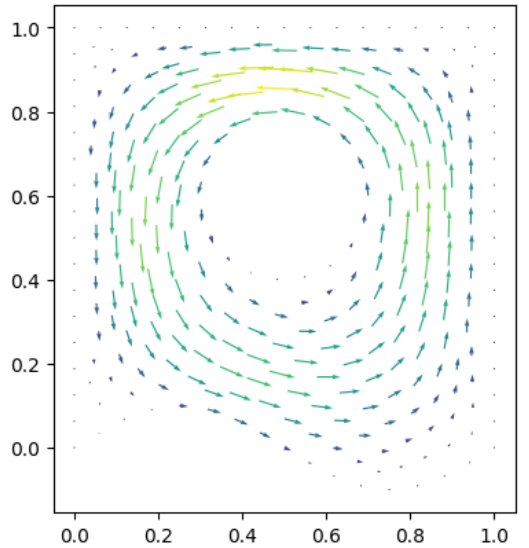}
     \caption{Velocity field for Example \ref{sec:ex}.2}
     \label{u_E2}
   \end{minipage}\hfill
   \begin{minipage}{0.48\textwidth}
     \centering
     \includegraphics[width=.7\linewidth]{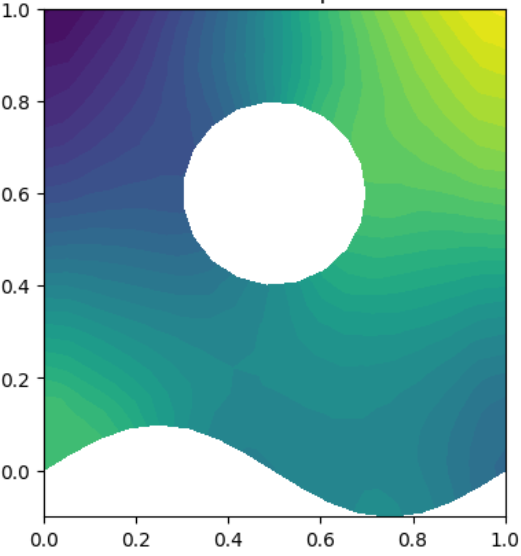}
     \caption{Pressure isobars for Example \ref{sec:ex}.2}
     \label{p_E2}
   \end{minipage}
\end{figure}

\begin{figure}
    \centering
    \includegraphics[width=0.5\linewidth]{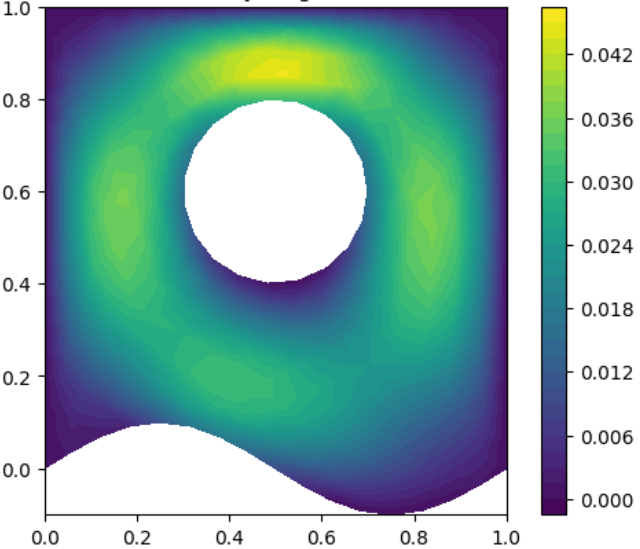}
    \caption{Magnitude of the velocity $|\bu^{h}|$ for Example \ref{sec:ex}.2}
    \label{VelocityMag_E2}
\end{figure}

\begin{figure}[h!]
   \begin{minipage}{0.48\textwidth}
     \centering
     \includegraphics[width=.7\linewidth]{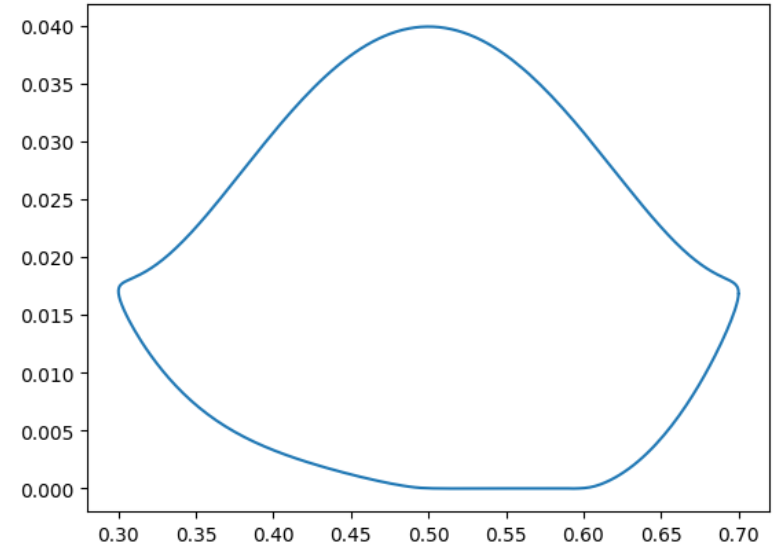}
     \caption{$ |\bu_{\tau}| $ along $\Gamma_{S,1}$, Example \ref{sec:ex}.2}
     \label{u_tau1_E2}
   \end{minipage}\hfill
   \begin{minipage}{0.48\textwidth}
     \centering
     \includegraphics[width=.7\linewidth]{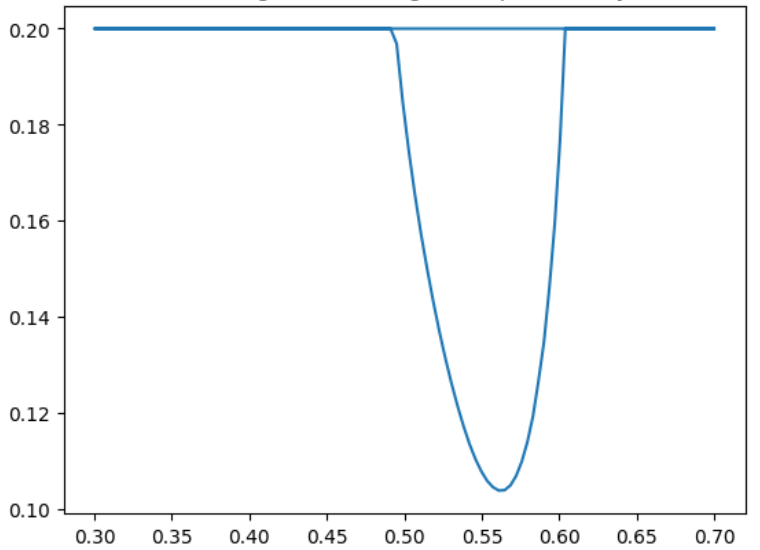}
     \caption{$ |\bsigma_{\tau}| $ along $\Gamma_{S,1}$, Example \ref{sec:ex}.2}
     \label{sigma_tau1_E2}
   \end{minipage}
\end{figure}

\begin{figure}[h!]
   \begin{minipage}{0.48\textwidth}
     \centering
     \includegraphics[width=.7\linewidth]{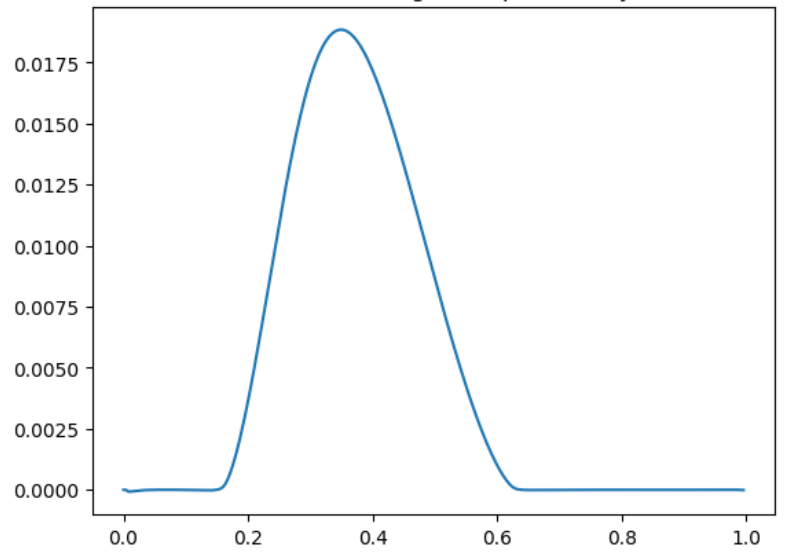}
     \caption{$ |\bu_{\tau}| $ along the slip boundary$\Gamma_{S,2}$}
     \label{u_tau2_E2}
   \end{minipage}\hfill
   \begin{minipage}{0.48\textwidth}
     \centering
     \includegraphics[width=.7\linewidth]{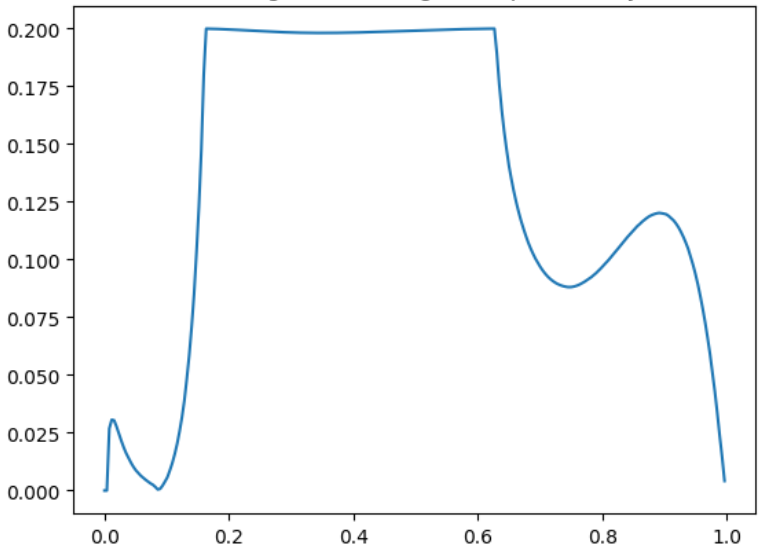}
     \caption{$ |\bsigma_{\tau}| $ along $\Gamma_{S,2}$, Example \ref{sec:ex}.2}
     \label{sigma_tau2_E2}
   \end{minipage}
\end{figure}

\end{document}